\documentclass[12pt]{article}
\usepackage[english]{babel}
\usepackage{amsmath, amsfonts, amsthm}
\usepackage{enumerate}
\usepackage[utf8]{inputenc}
\usepackage{xspace}
\usepackage[center]{subfigure}
\usepackage[justification=centering]{caption}
\usepackage{graphicx}
\usepackage{hyperref}
\usepackage{a4wide}
\usepackage{ifthen}
\usepackage{color}
\usepackage{caption}
\usepackage[vlined, titlenumbered, algoruled]{algorithm2e}
\usepackage{enumitem}
\usepackage{upgreek}
\usepackage[normalem]{ulem}

\definecolor{mygreen}{RGB}{46,139,87}

\usepackage{wasysym}

\SetAlgoSkip{smallskip}
\SetAlgoInsideSkip{smallskip}
\SetCommentSty{textit}
\SetAlCapNameSty{textsc}
\SetFuncSty{textsc}
\SetArgSty{textrm}
\SetDataSty{textsf}
\SetKwInput{KwSide}{Side effect}
\SetKwComment{Comment}{$\vartriangleright~$}{}
\DontPrintSemicolon
\SetSideCommentRight

\newtheorem{thm}{Theorem}[section]
\newtheorem{prop}[thm]{Proposition}
\newtheorem{lem}[thm]{Lemma}
\newtheorem{claim}[thm]{Claim}
\newtheorem{cor}[thm]{Corollary}
\newtheorem{defn}[thm]{Definition}

\newtheorem*{ack*}{Acknowledgment}
\theoremstyle{remark}
\newtheorem{rem}[thm]{Remark}

\newtheorem*{ex*}{Example}
\def\ie{\emph{i.e.}\xspace}
\def\eg{\emph{e.g.}\xspace}

\newcommand{\NN}{\mathbb{N}\xspace}
\newcommand{\ZZ}{\mathbb{Z}\xspace}

\makeatletter
\let\@@mathopen=\[
\let\@@mathclose=\]
\def\[{\@ifnextchar[{\@@doubleouvrant}{\@@mathopen}}
\def\@@doubleouvrant#1{[\![}
\def\@@doublefermant#1{]\!]}
\def\]{\@ifnextchar]{\@@doublefermant}{\@@mathclose}}
\makeatother
\DeclareMathOperator{\out}{out}
\DeclareMathOperator{\indeg}{in}
\newcommand{\QQ}[1]{\ensuremath{\mathcal{Q}_{#1}}\xspace}
\newcommand{\RR}[1]{\ensuremath{\mathcal{R}_{#1}}\xspace}
\newcommand{\TTe}[1]{\ensuremath{\mathcal{T}_{#1}}\xspace}
\newcommand{\CC}[1]{\ensuremath{\mathcal{C}_{#1}}\xspace}
\newcommand{\BB}[2][]{\ensuremath{\mathcal{B}_{#2}^{\ifthenelse{\equal{#1}{}}{}{(#1)}}}\xspace}
\newcommand{\TCP}{tree-\-and-\-closure partition\xspace}
\newcommand{\td}[1]{\ensuremath{t^{(d)}_{#1}}\xspace}
\newcommand\TT[1]{\ensuremath{\mathcal{T}_{#1}}\xspace}
\newcommand\MM[1]{\ensuremath{M{(#1)}}\xspace}
\def\Mr#1{\ensuremath{M_{\!/\!#1}}\xspace}

\def\Tdn#1{\mathsf{T}^{(#1)}\xspace}
\def\tdi#1{\ensuremath{{T}_{#1}}\xspace}
\def\3#1{\mathcal{T}_{#1}\xspace}

\def\Mdd{\ensuremath{\mathsf{M_d}}\xspace}
\def\Mdp{\ensuremath{\mathsf{M_{d,p}}}\xspace}
\def \Fdp{\mathsf{\tilde{F}^p_d}\xspace}
\def\hdp#1#2{\ensuremath{h_{#1}^{(#2)}}\xspace}
\def\SS#1#2{\ensuremath{\mathsf{S^{(d)}_i}}\xspace}
\def\ddori{$\frac{d}{d-2}$-orientation\xspace}

\newcommand{\eref}[1]{(\ref{#1})\xspace}

\newcommand{\Tbip}{\ensuremath{\mathcal{T}}\xspace}
\newcommand{\Pbip}{\ensuremath{\mathcal{P}}\xspace}
\newcommand{\be}{\bar e\xspace}
\newcommand{\bb}{\bar b\xspace}

\newcommand{\w}[1]{\ensuremath{w^{\scriptscriptstyle{(#1)}}}\xspace}
\newcommand{\tw}[1]{\ensuremath{\tilde{w}^{\scriptscriptstyle{(#1)}}}\xspace}

\newcommand{\p}[1]{\ensuremath{p^{\scriptscriptstyle{(#1)}}}\xspace}
\newcommand{\tp}[1]{\ensuremath{\tilde{p}^{\scriptscriptstyle{(#1)}}}\xspace}
\setlist[itemize]{noitemsep, topsep=0pt}

\newcommand{\sss}{\scriptscriptstyle}
\graphicspath{{Dessins/}}

\title{Generic method for bijections \\between blossoming trees and planar maps}
\author{Marie Albenque\footnote{albenque@lix.polytechnique.fr} \qquad
Dominique Poulalhon\footnote{poulalhon@liafa.univ-paris-diderot.fr}\\
\small LIX -- CNRS, \'Ecole Polytechnique, France\\
\small LIAFA -- CNRS, Universit\'e Paris Diderot, France}
\date{}
\begin{document}
%%%%%%%%%%%%%%%%%%%%%%%%%%%%%%%%%%%%%%%%%%%%%%%%%%%%%%%%%%%%%%%%%%%%%
%%%%%%%%%%%%%%%%%%%%%%%%%%%%%%%%%%%%%%%%%%%%%%%%%%%%%%%%%%%%%%%%%%%%%
\maketitle
%%%%%%%%%%%%%%%%%%%%%%%%%%%%%%%%%%%%%%%%%%%%%%%%%%%%%%%%%%%%%%%%%%%%%

 \begin{abstract}
  This article presents a unified bijective scheme between planar maps
  and blossoming trees, where a blossoming tree is defined as a
  spanning tree of the map decorated with some dangling
  half-edges that enable to reconstruct its faces. Our method
  generalizes a previous construction of Bernardi by loosening its
  conditions of application so as to include \emph{annular maps},
  that is maps embedded in the plane with a root face different from
  the outer face.

  The bijective construction presented here relies deeply on the
  theory of $\alpha$-orien\-ta\-tions introduced by Felsner, and in
  particular on the existence of minimal and accessible orientations.
  Since most of the families of maps can be characterized by such
  orientations, our generic bijective method is proved to capture as
  special cases many previously known bijections involving blossoming
  trees: for example Eulerian maps, $m$-Eulerian maps, non-separable
  maps and simple triangulations and quadrangulations of a $k$-gon.
  Moreover, it also permits to obtain new bijective constructions for
  bipolar orientations and $d$-angulations of girth $d$ of a $k$-gon.

  As for applications, each specialization of the construction
  translates into enumerative by-products, either via a closed formula
  or via a recursive computational scheme. Besides, for every family
  of maps described in the paper, the construction can be implemented
  in linear time. It yields thus an effective way to encode or
  sample planar maps.

  In a recent work, Bernardi and Fusy introduced another unified
  bijective scheme; we adopt here a different strategy %than theirs
  which allows us to capture different bijections. These two
  approaches should be seen as two complementary ways of unifying
  bijections between planar maps and decorated trees. \end{abstract}

%%%%%%%%%%%%%%%%%%%%%%%%%%%%%%%%%%%%%%%%%%%%%%%%%%%%%%%%%%%%%%%%%%%%%
\section*{Introduction}
%%%%%%%%%%%%%%%%%%%%%%%%%%%%%%%%%%%%%%%%%%%%%%%%%%%%%%%%%%%%%%%%%%%%%
The enumeration of planar maps was initiated in the 60's with the
pioneering work of Tutte \cite{Tut63}. To obtain enumeration formulas
for planar maps, Tutte uses some properties about their decomposition
to write equations satisfied by their generating series. The equations
thus obtained are quite complicated, in particular some additional
parameters (known as \emph{catalytic variables}) have usually to be
introduced to write them. The work of Tutte is a computational
\emph{tour de force}, since he managed, in most cases, to solve these
equations and to obtain closed (and particularly simple) formulas for
numerous families of maps.

This method turned out to be extremely versatile and can be applied to
many different models with only slight modifications. Furthermore, the
structure of the equations and of their solutions is now better
understood and some standard methods such as the ``quadratic method''
\cite[sec.2.9]{GouJac} and its extensions \cite{BouJeh06} are
available to compute them in standard cases. Tutte's approach, however,
gives little insight about the combinatorial structure of maps and in
particular produces no explanation of the reason why the formulas
obtained should be so simple.

Since then, successful ideas have been used to rederive and generalize
the results of Tutte, including work based on matrix integrals
\cite{Hooft74,BIPZ78}, algebraic approach \cite{Jackson87} and
bijective constructions \cite{CorVau81,Sch98}. The latter yield an
elegant way to rederive the formulas of Tutte but they also provide
tools for a better understanding of the combinatorial structure of maps: they
produce for example an efficient way to encode them by words of algebraic languages and hence to sample them efficiently \cite{Sch98,PouSch06}. This
led to establish new conjectures about the asymptotic behaviour of
random maps, which gave birth to a new field of research that has been
extremely active since (see for
instance~\cite{ChaSch04,MarMok06,LeGall07,Miermont13,LeGall13}). It is
noteworthy that a key ingredient in all these works about the
convergence of random planar maps is Schaeffer's bijection
\cite{Sch98} between quadrangulations and well-labelled trees, where
the distance in the quadrangulation between a vertex and the root is
encoded by the label of the corresponding vertex in the tree, or one
of its generalizations \cite{BDG04,ChaMarSch09,Miermont09}.

\bigbreak

Let us now focus on those bijective proofs. Formulas obtained by Tutte and its successors suggest that maps could be interpreted as trees with some decorations. After initial work in that direction by Cori and Vauquelin \cite{CorVau81}, Schaeffer \cite{Sch98} drew new attention on this field by obtaining numerous bijective constructions.
This founding work was followed by a series of papers dealing with
various families of maps: a non-exhaustive list includes maps with
prescribed degree sequence \cite{Schaeffer97, BDG02,BDG04}, maps
endowed with a physical model \cite{MBMSch00,BDG07} or with
connectivity constraints \cite{PouSch06}. Each of these bijections
appears as an ad-hoc explanation of the known enumeration formula, but
they present strong similarities, which calls for a unified bijective
theory. An important step in that direction has been achieved
in~\cite{BerFusPentagulation,BerFus12}, where a ``master bijection''
is introduced in order to see many constructions as special cases of a
common construction. The main purpose of the present paper is to
present a different attempt in unifying the bijective constructions, in
particular so as to include some bijections that are not captured by
the work of Bernardi and Fusy.

Allow us to be slightly more precise. The first bijections obtained
rely on the existence of a canonical spanning tree of the map
\cite{Schaeffer97, Sch98, BDG02} or of its quadrangulation
\cite{CorVau81, Sch98, BDG04}.  As emphasized by
Bernardi~\cite{Bernardi07}, a map endowed with a spanning tree can
also be viewed as a map endowed with an orientation of its edges with
specific properties. The latter point of view appears to be more
suitable to unify and generalize the constructions.  In particular,
the master bijection defined in~\cite{BerFusPentagulation} and
\cite{BerFus12} is based on a generalized notion of orientations. This
construction includes as special cases many previously known
bijections, but unfortunately not all of them and in particular not
the case of simple triangulations \cite{PouSch06} and quadrangulations
\cite[ch.3]{FusyThesis}. We would like to emphasize that in both these
cases, the tree obtained is simply a spanning tree of the original
map. The master bijection, in contrast, produces a spanning tree of
the superimposition of a planar map, its dual and their (common)
quadrangulation.

% This is maybe the only drawback of the very general setting
%introduced in these papers.

\bigbreak

The ground result of our paper is to present a new bijective scheme
which relies on an orientation of the map and yields a spanning tree
of the map with some decorations that allow to reconstruct facial
cycles. It generalizes the result of~\cite{Bernardi07} by loosening
the rooting conditions. In particular it enables to deal with
\emph{annular maps} (that is, rooted planar maps with a marked face)
such as triangulations of a $p$-gon \cite{PouSch06}. Notably all the
previous bijective constructions that involve a spanning tree of the
map are captured by our generic scheme and, moreover, we obtain new
bijections for plane bipolar orientations and $d$-angulations of a
$p$-gon with girth $d$. Besides, the first bijective proof of a
well-known theorem by Hurwitz on products of transpositions in the
symmetric group has been obtained recently by Duchi, Schaeffer and the
second author~\cite{DucPouSch13} using this generic scheme.

Bijective proofs appear often as an a posteriori enlightening
explanation of a simple enumerative formula. In fact, the formula is
used as a guide to construct the ``simplest'' objects that it
enumerates: the right objects to consider can be seen as its
combinatorial translation. Here, remarkably, the satisfying
orientations are often natural enough so that they can be guessed even if a
formula is not available. 

Another important feature of our generic scheme is its constructive
character: given a blossoming tree, the corresponding map can be
computed in linear time. Reciprocally given a map endowed with the
appropriate orientation, the corresponding blossoming tree can be
computed in quadratic time by a generic algorithm. In fact, for all
the families of maps considered in this paper, ad-hoc algorithms can
be designed to compute the blossoming tree in linear time. This was
known to be true for rooted maps~\cite{Bernardi07}, and one of our
main contributions is to design a linear-time algorithm that computes
the blossoming tree of a $d$-angulation of a $p$-gon.

\bigbreak

To conclude, let us mention three perspectives to continue this work.
The bijective method we develop relies deeply on orientations and,
algorithmically speaking, takes as input a map endowed with a specific
orientation. An algorithm by Felsner~\cite{Felsner04} (see also
\cite[p.56]{FusyThesis}) ensures that, for a fixed map with $n$ vertices and a
prescribed sequence $\alpha$ of outdegrees, an $\alpha$-orientation
(if any exists) can be computed by a generic algorithm of complexity
$\Uptheta(n^{3/2})$, if the map has bounded maximal
outdegree.  For various families of maps, linear time algorithms do
exist, but it is still an open problem to design such an algorithm for
$d$-angulations when $d\geq 5$ (unlike simple triangulations and
quadrangulations, see for instance \cite[ch.2]{FusyThesis}).

Secondly, almost all models of maps appear now as special cases either
of our generic scheme or of the master bijections of Bernardi and
Fusy.  Nevertheless, a few models are still not captured; this in
particular the case of models with ``matter'' such as the Ising model,
for which some bijections with blossoming trees are known
\cite{MBMSch00}. Some additional work is needed to either generalize
one of those schemes or to come up with an alternative approach.

Lastly, as mentioned above, the scaling limit of random plane maps has
been a very active area of research in the last years. So far, it has
been proved that for $p=3$ or $p$ even, the limit of (properly
rescaled) $p$-angulations is the so-called ``Brownian map''
\cite{Miermont13,LeGall13}. It is widely believed that all the
reasonable families of maps -- which includes for instance
$p$-angulations for $p$ odd or maps with constraints on their
connectivity -- belong to the same universality class, or in other
words, should converge to the same limit object.  A first result in
this direction about simple triangulations and quadrangulations has
been obtained very recently by Addario-Berry and the first
author~\cite{AdAl13}. The proof of their result relies on the
bijections of~\cite{PouSch06} and~\cite{FusyThesis} and on a way to
interpret the distances in the map on the
corresponding blossoming tree. It would be a major breakthrough to
generalize their result to all the maps captured by our scheme.

\paragraph{Outline} In Section~\ref{sec:MapOri}, we gather definitions
about maps and their orientations and recall the fundamental result of
Felsner about uniqueness of minimal $\alpha$-orientations
(Theorem~\ref{th:Felsner}). We introduce \emph{blossoming maps} in
Section~\ref{sec:main} so as to describe and prove our bijective scheme
along with some remarks about its complexity.

In Section~\ref{sec:PrevBij}, previous bijections obtained for
Eulerian maps and general maps (Subsection~\ref{sub:degseq}),
$m$-Eulerian maps (Subsection~\ref{sub:mEul}) and non-separable maps
(Subsection~\ref{sub:nonsep}) are rederived via our bijective
technique. A new bijection between bipolar orientations and some
triples of paths is obtained in Section~\ref{sec:bipolar}.

Sections~\ref{sec:dangulations} and~\ref{sec:opening} are devoted to
$d$-angulations of a $p$-gon. More precisely,
Section~\ref{sec:dangulations} describes {\ddori}s and the bijection
between $p$-gonal $d$-fractional forests and $p$-gonal $d$-angulations
as well as enumerative consequences, while Section~\ref{sec:opening}
focuses on the description and the proof of the linear time opening
algorithm in that setting.

\paragraph{Acknowledgments.} 
We would like to thank Éric Fusy and Gilles Schaeffer for fruitful
discussions and a referee whose careful reading helped us to improve significantly the exposition of our work. 
We acknowledge the
support of the ERC under the agreement ``ERC StG 208471 - ExploreMap'' and of the 
ANR under the agreement ``ANR 12-JS02-001-01 - Cartaplus''. 
%%%%%%%%%%%%%%%%%%%%%%%%%%%%%%%%%%%%%%%%%%%%%%%%%%%%%%%%%%%%%%%%%%%%%

%%%%%%%%%%%%%%%%%%%%%%%%%%%%%%%%%%%%%%%%%%%%%%%%%%%%%%%%%%%%%%%%%%%%%
\section{Maps and orientations}\label{sec:MapOri}
%%%%%%%%%%%%%%%%%%%%%%%%%%%%%%%%%%%%%%%%%%%%%%%%%%%%%%%%%%%%%%%%%%%%%

%%%%%%%%%%%%%%%%%%%%%%%%%%%%%%%%%%%%%%%%%%%%%%%%%%%%%%%%%%%%%%%%%%%%%
\subsection{Planar maps}\label{sec:maps}

A \emph{planar map} is a proper embedding of a connected graph in the
sphere, where \emph{proper} means that edges are smooth simple arcs
which meet only at their endpoints. Two planar maps are identified if
they can be mapped one onto the other by a homeomorphism that
preserves the orientation of the sphere. \emph{Edges} and
\emph{vertices} of a map are the natural counterparts of edges and
vertices of the underlying graph. The \emph{faces} of a map are the
connected components of the complementary of the embedded graph. The
embedding fixes the cyclical order of edges around each vertex, which
defines readily a \emph{corner} as a couple of consecutive edges
around a vertex (or, equivalently, around a face).  Corners may also
be viewed as incidences between vertices and faces. The \emph{degree}
of a vertex or a face is defined as the number of its corners. In
other words, it counts incident edges with multiplicity 2 for each
loop (in the case of vertex degree) or for each bridge (in the case of
face degree).

In these definitions, vertices and faces play a similar role. It is
often useful to exchange them and to consider the \emph{dual}
$M^\star$ of a given map $M$, whose vertices correspond to faces of
$M$ and faces to vertices of $M$. Edges are somehow unchanged: each
edge $e$ of $M$ corresponds to an edge of $M^\star$ that is incident to
the same vertices and faces as~$e$, see Fig.\ref{fig:dual}.

\begin{figure}[t]
  \centering
  \subfigure[\label{fig:dual}Dualization,]{\includegraphics[page=4, scale=0.64]{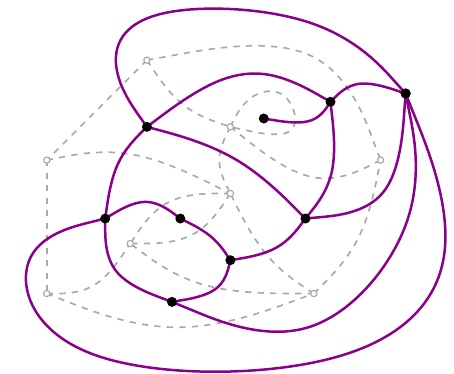}}\quad
  \subfigure[\label{fig:quad} quadrangulation,]{\;\includegraphics[page=3, scale =0.64]{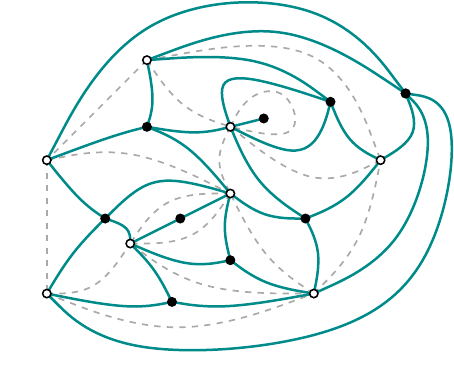}\;}\quad
  \subfigure[\label{fig:radial}and radialization.]{\;\includegraphics[page=2, scale =0.64]{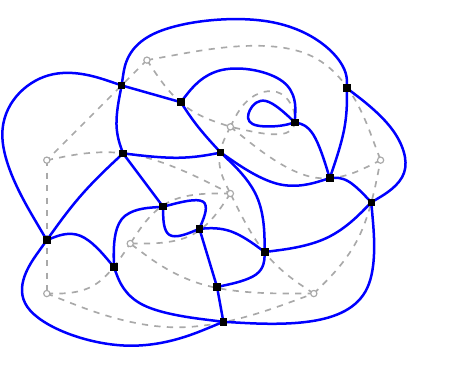}}
  \caption{Example of classical constructions on a map. The original map is drawn in dashed grey lines with white vertices. Its dual map, quadrangulation and radial map are drawn in plain colored lines.}
  \label{fig:Dual_Quad}
\end{figure}

A planar map is said to be \emph{$d$-regular} if all its vertices have
degree~$d$.  Dually, a planar map is called a \emph{$d$-angulation} if
all its faces have degree~$d$; the terms \emph{triangulation},
\emph{quadrangulation} and \emph{pentagulation} correspond
respectively to the cases where $d=3,4,5$.  The following classical
and useful construction associates a quadrangulation to each planar
map~$M$. Let us say that vertices of $M$ are white; add a black vertex
in each face of $M$, and an edge in each corner of $M$ between the
corresponding white and black vertices. This produces a triangulation
with bicolored vertices. Keeping only the additional edges leads to a
quadrangulation $\QQ M$ which is called \emph{the} quadrangulation
of~$M$, see Fig.\ref{fig:quad}. Its dual map $\RR M$ is called the
\emph{radial map} of~$M$, see Fig.\ref{fig:radial}.

A \emph{plane map} is a proper embedding of a connected graph in the
plane; its unique unbounded face is called the \emph{outer face}, and all
the other faces are called \emph{inner} faces. Vertices and edges are
called \emph{outer} or \emph{inner} depending on whether they are
incident to the outer face or not. Observe that a plane map is in fact
a planar map with a distinguished marked face.
 
A planar or plane map is said to be \emph{(corner)-rooted} if one of
its corners is distinguished. The corresponding vertex and face are
called \emph{root vertex} and \emph{root face}. In the figures, the
root corner is indicated by a double arrow and the root vertex by a
square vertex, see Fig.\ref{fig:2-fractional} for instance. The
\emph{root edge} is defined as the edge that follows the root corner
in clockwise order around the root vertex. The usual convention is to
associate to each rooted planar map the rooted plane map in which the
root and the outer faces coincide. However in this work, plane maps
are allowed to have one root face different from the outer face (in
the literature, planar maps for which the root face is different from
the outer face are sometimes called \emph{annular maps}). Some weaker
rootings will sometimes be also considered by only pointing either a
root vertex, a root edge or a root face. In the latter case, vertices
or edges incident to the root face are called \emph{root vertices} or
\emph{root edges}.

A \emph{plane tree} is a planar map with a single face; its vertices
are called \emph{nodes} and \emph{leaves} depending on their degree
being greater than or equal to one.  A \emph{planted tree} is a plane
tree rooted at a leaf. Observe that usual ``planar trees'' or
``ordered trees'' are obtained from planted trees by deleting their
root leaf.

%%%%%%%%%%%%%%%%%%%%%%%%%%%%%%%%%%%%%%%%%%%%%%%%%%%%%%%%%%%%%%%%%%%%%
\subsection{Orientations}\label{sec:orientation}
This section gathers definitions and fundamental results about
orientations of planar maps. The terminology and convention are not
completely standard and we emphasize the differences when needed.

An \emph{orientation} of a map is the choice of a direction for each
of its edges.  The \emph{indegree} or \emph{outdegree} of a vertex
$v$, denoted $\indeg(v)$ or $\out(v)$, is the number of edges oriented
inwards or outwards~$v$.  Let $M$ be a planar map, $V$ the set of its
vertices, and let $\alpha: V \rightarrow \NN$ be an application which
associates a natural number to each vertex of the map. An
\emph{$\alpha$-orientation} -- as introduced by Felsner
in~\cite{Felsner04} -- is an orientation of $M$ such that for each
vertex $v$ in $V$, $\out(v)=\alpha(v)$. If such an orientation exists,
$\alpha$ is said to be \emph{feasible}. An orientation of a
corner-, vertex- or face-rooted map is said to be \emph{accessible} if for
any vertex $v$, there exists an oriented path (see below) from $v$ to the root
vertex (or to one of the root vertices, in the case of a face
rooting).

\medskip

An \emph{oriented path} is an alternating sequence
$(v_0,e_1,v_1,\dots,v_{\ell-1},e_\ell,v_\ell)$ of incident vertices
and edges in which each edge $e_i$ is oriented from $v_{i-1}$ to
$v_i$.  An \emph{oriented cycle} (also called a \emph{circuit}) is
defined accordingly.  The (canonical) embedding of plane maps enables
to define \emph{clockwise} and \emph{counterclockwise cycles} as
simple oriented cycles with the outer face respectively on their left
and on their right. Observe that the orientation obtained after
reverting all the edges of a given oriented cycle is still an
$\alpha$-orientation. In fact, all the $\alpha$-orientations of a map
$M$ can be obtained by a sequence of such \emph{flips},
see~\cite{Felsner04}. In particular, it implies that either all or
none $\alpha$-orientations of $M$ are accessible. In the former case,
$\alpha$ is said \emph{accessibly feasible}. Moreover:
%
%\begin{thm}[Felsner~\cite{Felsner04}]\label{th:Felsner}
%  Let $M$ be a plane map and $\alpha$ be a feasible function on its
%  vertices.  The set of $\alpha$-orientations of $M$ can be endowed
%  with a lattice structure, whose minimal element is the unique
%  $\alpha$-orientation of $M$ without counterclockwise cycles.
%\end{thm}
%
%The main consequence of this theorem for our purpose is to associate
%canonically to any given feasible $(M,\alpha)$ the (unique) minimal
%$\alpha$-orientation of $M$. From now on, we call \emph{minimal} any
%orientation without counterclockwise cycles.
%

\begin{thm}[Felsner~\cite{Felsner04}]\label{th:Felsner}
Let $M$ be a plane map and $\alpha$ be a feasible function on its vertices. Then, there exists a unique $\alpha$-orientation without counterclockwise cycles. 
\end{thm}

The relevance of this theorem for our purpose is to associate \emph{canonically} to any given feasible $(M,\alpha)$ one specific $\alpha$-orientation. From now on, we call \emph{minimal} any orientation without counterclockwise cycles.

\begin{rem}
Let us mention that the result of Felsner is in fact much stronger. He proves indeed that the set of $\alpha$-orientations of $M$ can be endowed with a lattice structure, where the cover relation corresponds essentially to the flip of an oriented cycle. The (unique) minimum element of this lattice is the minimal $\alpha$-orientation, hence its name.
\end{rem}

\begin{figure}
  \centering
  \subfigure[A 2-fractional oriented map (with counterclockwise
  cycles),]{\quad\includegraphics[page=1, scale =1.2]{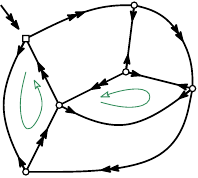}\quad}\qquad
  \subfigure[its representation as an oriented 2-expanded map,]{\;\includegraphics[page=2, scale =1.2]{orientation_fractionnaire}\;}\qquad
  \subfigure[and the minimal orientation with same in- and out-degrees.]{\quad\;\includegraphics[page=3, scale =1.2]{orientation_fractionnaire}\quad\;}
  \caption{A rooted plane map endowed with accessible 2-fractional orientations.}
  \label{fig:2-fractional}
\end{figure}

\medskip The $k$-expanded version of a plane map $M$ is defined as the
plane map where each edge of $M$ has been replaced by $k$ copies.  A
\emph{$k$-fractional orientation} of $M$ is defined
in~\cite{BerFusPentagulation} as an orientation of the $k$-expanded
map of $M$, with the additional property that two copies of the same
edge cannot create a counterclockwise cycle. It is conveniently
considered as an orientation of $M$ in which edges can be partially
oriented in both directions and the in- or out-degree of a vertex $v$
(that can now be fractional) is defined as the in- or out-degree of $v$ in
the $k$-expanded map, divided by~$k$.  In this setting, a
\emph{saturated edge} is an edge which is totally oriented in the same
direction, an \emph{oriented path} is a path in which each edge is at
least partially oriented in the considered direction. The notions of
clockwise or counterclockwise cycles, of minimality and of accessibility
follow. Figure~\ref{fig:2-fractional} shows the two possible
representations of a given 2-fractional orientation, with two counterclockwise cycles, and the corresponding minimal
2-fractional orientation.
%%%%%%%%%%%%%%%%%%%%%%%%%%%%%%%%%%%%%%%%%%%%%%%%%%%%%%%%%%%%%%%%%%%%%

%%%%%%%%%%%%%%%%%%%%%%%%%%%%%%%%%%%%%%%%%%%%%%%%%%%%%%%%%%%%%%%%%%%%%
\section{A generic bijective scheme for maps endowed with a minimal orientation}\label{sec:main}
%%%%%%%%%%%%%%%%%%%%%%%%%%%%%%%%%%%%%%%%%%%%%%%%%%%%%%%%%%%%%%%%%%%%%

%%%%%%%%%%%%%%%%%%%%%%%%%%%%%%%%%%%%%%%%%%%%%%%%%%%%%%%%%%%%%%%%%%%%%
\subsection{Blossoming maps and closure}\label{sec:blossom}
%%%%%%%%%%%%%%%%%%%%%%%%%%%%%%%%%%%%%%%%%%%%%%%%%%%%%%%%%%%%%%%%%%%%%

\begin{defn}\label{def:contourword}
  A \emph{blossoming map} is a plane map, in which each outer corner
  can carry a sequence of opening or closing \emph{stems}
  (in the literature, opening and closing stems are sometimes referred
  to as buds and leaves).
  
  The \emph{cyclic contour word} of a blossoming map is the word on
  $\{e,b,\bb\}$, which encodes the cyclic clockwise order of edges and
  stems along the border of the outer face with $e$ coding for an edge
  and $b$ and $\bb$ for opening and closing stems, see Fig.\ref{fig:cycliccontour}.
\end{defn}

\begin{figure}
  \centering
    \subfigure[\label{fig:cycliccontour}A blossoming map, with~cyclic contour word
    $\underline{bee\bar{b}}ee\bar{b}beebee\bar{b}\bar{b}ee\bar{b}$,]{%
      \;\includegraphics[page=4, scale =1.2]{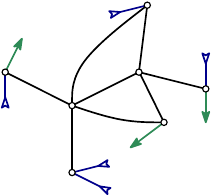}\;}\qquad
  \subfigure[\label{fig:localclosure}a local closure, which yields a new cyclic contour word
    $eee\bar{b}\underline{bee\underline{bee\bar{b}}\bar{b}}ee\bar{b}$,]{\;\includegraphics[page=5, scale=1.2]{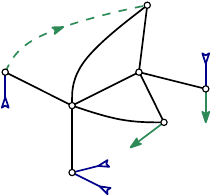}\;}\qquad
  \subfigure[\label{fig:totalclosure}and the total closure, with cyclic contour word $eee\bar{b}eee\bar{b}$.]{\;\includegraphics[page=6, scale=1.2]{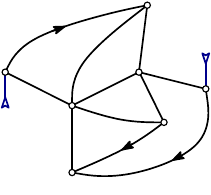}\;}
  \caption{Closure of a blossoming map. Opening stems are represented
    by plain green arrows, and closing stems by reverse blue
    arrows. Factors to be substituted by $e$ are underlined in the
    contour word.}
  \label{fig:closure}
\end{figure}

A \emph{local closure} of a blossoming map is a substitution of a
factor $be^{\star}\bb$ by the letter $e$ in its contour word, where
$e^\star$ denotes any sequence of $e$ (possibly empty), see
Fig.\ref{fig:localclosure}. In terms of maps, it corresponds to the
creation of a new edge (and hence a new face) by merging an opening
stem with the following closing stem (provided that there is no other
stem in between) in clockwise order around the border of the outer
face. The new edge is canonically oriented from the opening vertex to
the closing vertex, with the new bounded face on its right. If several
local closure operations are possible on a blossoming map, performing
all of them in either order yields the same result, hence iterating
such local closures produces eventually a unique object:

\begin{defn}\label{def:closure}
  The \emph{closure} $\overline M$ of a blossoming map $M$ is the
  (possibly blossoming) map obtained after iterating as many local
  closure operations as possible, see Fig.\ref{fig:totalclosure}. When
  only a subset of local closures is performed, the map obtained is
  called a \emph{partial closure} of~$M$.  The edges created during
  local closures operations are called \emph{closure edges}.
\end{defn}

In particular, the closure of a blossoming map with an equal number of
opening and closing stems is a (non-blossoming) map.  Since
closure edges are canonically oriented, if a blossoming map is endowed
with an orientation (possibly $k$-fractional), so is its
closure. Moreover, considering opening and closing stems respectively
as outgoing and incoming (half-)edges, in- and out-degrees are
preserved.  Since all the closures are performed in clockwise
direction around the map, no counterclockwise cycle can be created
during a local closure operation. Consequently if the initial
orientation is minimal, then so is its closure. Accessibility is
preserved as well.

The most interesting special case is the one of a rooted plane tree,
endowed with an accessible orientation, which, in the classical
non-fractional setting, implies that edges are oriented towards the
root vertex.  Examples of a blossoming tree and of its closure are
given in Figs.~\ref{fig:blossom} and~\ref{fig:closureblossom}. See
also Fig.~\ref{fig:non-Eulerian-blossom}
and~\ref{fig:non-Eulerian-closure} for a 2-fractional example.

The aim of the next section is to provide an inverse construction of the
closure.

%%%%%%%%%%%%%%%%%%%%%%%%%%%%%%%%%%%%%%%%%%%%%%%%%%%%%%%%%%%%%%%%%%%%%
\subsection{Orientations and opening}\label{sec:edge-partition}
%%%%%%%%%%%%%%%%%%%%%%%%%%%%%%%%%%%%%%%%%%%%%%%%%%%%%%%%%%%%%%%%%%%%%
The following theorem generalizes a result on tree orientations that can
be explicitly found \eg in~\cite{Bernardi07}, and which is at
the heart of all bijections between map orientations and blossoming
trees.

\begin{thm}\label{thm:open}
  Let $M$ be a plane map vertex-rooted at
  $r$, 
  and suppose that $M$ is endowed with a minimal accessible
  orientation 
  $O$. Then $M$ admits a unique edge-partition $(\TT M,\CC M)$ such
  that:

  \begin{itemize}
  \item edges in $\TT M$ (called \emph{tree edges}) form a spanning tree
    of $M$, rooted at $r$, on which the restriction of $O$ is
    accessible;
  \item any edge in $\CC M$ (called a \emph{closure edge}) is a
    saturated clockwise edge in the unique cycle it forms with edges
    in $\TT M$.
  \end{itemize}
  Let us call such a partition a \emph{\TCP}.
\end{thm}

\begin{figure}[t]
  \centering
  \subfigure[A minimal accessible
  orientation,\label{fig:closureblossom}]{\;\includegraphics[page=3,
    scale =1.2]{blossoming_to_minimal}\;}\qquad
  \subfigure[the corresponding \TCP,]{\;\includegraphics[page=2, scale=1.2]{blossoming_to_minimal}\;}\qquad
  \subfigure[and the corresponding blossoming
  tree.\label{fig:blossom}]{\;\includegraphics[page=1, scale=1.2]{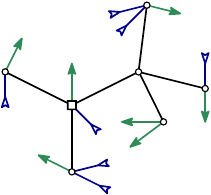}\;}
 \caption{From a minimal accessible orientation to a blossoming tree.}
  \label{fig:minacc_to_blossoming}
\end{figure}

Before proving this theorem, we would like to emphasize in which way
it differs from the result obtained in~\cite{Bernardi07}. In the
latter work, the outer face of a rooted map is required to be the root
face.  In this case and also in the particular case of triangulations
treated in~\cite{PouSch06}, a contour algorithm, starting at the root
edge, enables to identify the edges of $\CC M$. The proof that this
algorithm is correct relies deeply on the fact that \emph{both} the
accessibility and the minimality of the orientation are defined
according to the root face. We show here that this hypothesis is
unnecessary.

\begin{figure}[t]
  \centering \subfigure[$e$ is a closure edge of this
  map. \label{fig:induction}]{\includegraphics[height=10em]{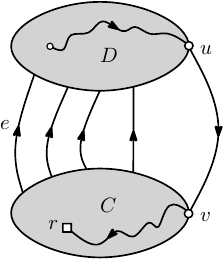}}
  \qquad\qquad \subfigure[Proof of the uniqueness: the path $\gamma$
  in $\TT M$ is~represented in fat plain green edges, and the~tree~$\TT
  M'$ in dashed blue
  edges.\label{fig:uniqueness}]{\includegraphics[height=8em]{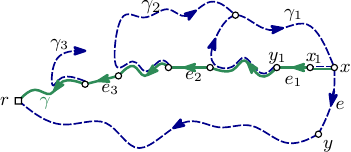}}
  \caption{Existence of a unique \TCP.}
  \label{fig:TCP}
\end{figure}

\begin{proof}
  We prove this result by induction on the number of faces of $M$. If
  $M$ has only one face, it is an accessible tree, hence the property
  is satisfied.

  Let now $n\geq2$, and suppose that the property is satisfied for any
  minimally oriented plane map with less than $n$ faces. Let $M$ be a
  vertex-rooted plane map with $n$ faces endowed with a minimal
  accessible orientation.

  We shall first prove that one of its outer edges $e$ may be removed
  to obtain a vertex-rooted map $M_{\setminus e}$ endowed with a
  minimal accessible orientation. As outer edges do not form a
  counterclockwise cycle, at least one of them has the outer face
  strictly on its left, meaning that it is saturated and is not a
  bridge. Let $(u,v)$ be such an edge and consider the map
  $M_{\setminus (u,v)}$. If it is accessible, then choose $e =
  (u,v)$. Otherwise, let $C$ be the accessible component of the root
  vertex $r$ in $M_{\setminus (u,v)}$. Then clearly $v$
  belongs to $C$ while $u$ does not. Moreover, $u$ is accessible from
  all vertices not in $C$, see Fig.~\ref{fig:induction}.
%% Let $D$ be the accessible component of $u$. 
 Let $D$ denote the complement of~$C$.
  Then, the cut between $C$ and $D$, made up of saturated edges
  oriented from $C$ to $D$, is incident twice to the outer face of
  $M_{\setminus (u,v)}$. Let $e$ be the edge of the cut with the outer
  face on its left. Since $e$ is not a bridge in $M$, the map
  $M_{\setminus e}$ has $n-1$ faces and the orientation induced by
  that of $M$ is minimal and accessible.

  Hence, by induction, $M_{\setminus e}$ admits a (unique) \TCP $(\TT
  M, \CC M)$, and $(\TT M, \CC M \cup \{e\})$ is a \TCP for $M$.

  Let us now prove that $M$ does not admit any other \TCP, that is,
  does not admit any \TCP with $e$ in the tree. Suppose by
  contradiction that $(\TT M', \CC M')$ is a \TCP for $M$ with $e \in
  \TT M'$. Let us denote $e = (x,y)$, oriented from $x$ to $y$, and
  consider the simple path $\gamma$ from $x$ to the root vertex in
  $\TT M$. At least one edge of $\gamma$ does not belong to $\TT M'$
  (otherwise this would create a cycle in the tree), hence $\gamma$
  contains saturated edges. Let $e_1 =(x_1, y_1)$ be the first
  saturated edge in $\gamma$. Since $x_1$ lies in the subtree $\TT
  M'(x)$ of $\TT M'$ rooted at $x$, $e_1$ belongs to $\CC M'$. Hence
  $y_{1}$ is explored after $x_{1}$ in the clockwise contour of $\TT
  M'$. Moreover, because $(x,y)$ has the outer face on its left, the
  path $\gamma_{1}$ from $y_{1}$ to the root vertex of $\TT M'$ cannot
  wrap $(x,y)$. Moreover, because $(x,y)$ has the outer face on its
  left, it cannot be wrapped by a closure edge.  Hence $y_1$ belongs
  to $\TT M'(x)$, see Fig.~\ref{fig:uniqueness}.  In particular, there
  exists another (saturated) edge $e_2$ of $\gamma$ that does not
  belong to $\gamma_1$ nor $\TT M'$ for which the same reasoning
  applies. This implies the existence of an infinite sequence of edges
  of $\gamma$ not belonging to $\TT M'$, a contradiction.
\end{proof}

As an immediate corollary of Theorem~\ref{thm:open}, we obtain:

\begin{cor}\label{cor:blossom}
  Let $M$ be a vertex-rooted plane map endowed with a minimal
  accessible orientation $O$. Then there exists a unique vertex-rooted
  blossoming tree, endowed with an accessible orientation, the closure
  of which is~$M$ oriented with $O$.

  This blossoming tree is denoted \BB{M}.
\end{cor}

\begin{proof}
Let $M$ be a vertex-rooted plane map endowed with a minimal accessible
orientation~$O$. Consider the \TCP $(\TT M, \CC M)$ of $M$. The edges
of the blossoming tree $\BB{M}$ are the edges of $\TT M$ and each edge
of $\CC M$ is cut in two to produce a pair of opening and closing stems (see Fig.\ref{fig:minacc_to_blossoming}).
\end{proof}

Hence, in any particular case where a family of plane maps may be
canonically endowed with a family of specific minimal and accessible
orientations, Corollary~\ref{cor:blossom} gives a bijection between
that family of maps and a family of blossoming trees with the same
distribution of in- and out- degrees. Whenever these trees are easily
described and enumerated, this leads to a bijective proof of
enumerative results.

%%%%%%%%%%%%%%%%%%%%%%%%%%%%%%%%%%%%%%%%%%%%%%%%%%%%%%%%%%%%%%%%%%%%%
\subsection{Effective opening and closure}\label{sec:effective}
%%%%%%%%%%%%%%%%%%%%%%%%%%%%%%%%%%%%%%%%%%%%%%%%%%%%%%%%%%%%%%%%%%%%%
Let us point out some facts about the complexity of computing effectively the
closure of a blossoming tree and the opening of a map. 

To close a blossoming tree into a map, it is enough to perform a
contour process and to match iteratively each opening stem with its
corresponding closing stem along the way. Each time a new opening stem
is explored, it can be stored in a stack structure
(Last-In-First-Out), out of which one stem will be popped each time a
closing stem is explored. This process goes around the outer face at
most twice, hence the total time complexity is linear in the number of
edges of the final map.

Unfortunately, things are not so smooth when it comes to opening an
oriented plane map into its blossoming tree. Since the proof of
Theorem~\ref{thm:open} is essentially constructive, it yields an
algorithm which identifies a closure edge at each step. Each of these
steps consists in computing an accessible component, which can be done
in linear time, resulting in a total quadratic complexity.

However, in the case where the map is corner-rooted in the outer face,
the opening operation can be realized in linear time by an adapted
depth-first search process. This construction has been introduced in a
series of papers (see for example \cite{Sch98,PouSch06}) in some
particular cases and formally stated in~\cite{Bernardi07} (where it
appears in a slightly different form since the convention for tree
edges orientation is opposite to ours).

\begin{prop}[\cite{Sch98,PouSch06,Bernardi07}]\label{prop:bernardi}
  Let $M$ be a corner-rooted plane map in which the outer and root
  faces coincide, and assume that $M$ is endowed with a minimal
  orientation.

  Then, the \TCP of $M$ can be computed in linear time: initialize $\TT M$
  and $\CC M$ as empty sets, and $v$ and $e$ to be respectively the root
  vertex and the root edge; then repeat the following steps until all
  edges belong either to $\CC M$ or to $\TT M$:

  \begin{itemize}
  \item if $e$ does not belong to $\CC M$ nor $\TT M$ yet, add it to
    $\CC M$ if it is oriented outwards $v$, and to $\TT M$ otherwise;
  \item if $e$ belongs to $\TT M$, switch $v$ to the other extremity
    of $e$;
  \item update $e$ to the next edge around $v$ in clockwise order.
   \end{itemize}
\end{prop} 

The proof of this proposition can be found in \cite{Bernardi07}, in a
different setting: it deals with rooted planar maps endowed with a
distinguished spanning tree, for which an orientation is canonically
defined by orienting tree edges towards the root, and any other
clockwise in the unique cycle it forms with the tree. The set of all
these \emph{tree orientations} is hence equal to the set of minimal
$\alpha$-orientations for all accessibly feasible $\alpha$.

Observe that as opposed to the purpose of \cite{Bernardi07}, our work
aims at defining one canonical orientation for each map. This requires
to seek for an appropriate family of functions $\alpha$ for each
family of maps we want to enumerate, as already mentioned at the end
of Section~\ref{sec:edge-partition}. In the next section, we
demonstrate that numerous previously known bijections can be easily
retrieved as soon as we exhibit the adequate $\alpha$.

%%%%%%%%%%%%%%%%%%%%%%%%%%%%%%%%%%%%%%%%%%%%%%%%%%%%%%%%%%%%%%%%%%%%%

%%%%%%%%%%%%%%%%%%%%%%%%%%%%%%%%%%%%%%%%%%%%%%%%%%%%%%%%%%%%%%%%%%%%%
\section{Recovering previous bijections}\label{sec:PrevBij}
%%%%%%%%%%%%%%%%%%%%%%%%%%%%%%%%%%%%%%%%%%%%%%%%%%%%%%%%%%%%%%%%%%%%%

Previous bijections between planar maps and blossoming trees obtained
after Schaeffer~\cite{Schaeffer97} can all be seen as applications of
Corollary~\ref{cor:blossom}, and more precisely of
Proposition~\ref{prop:bernardi}. In this section, we only consider
corner-rooted planar maps, that is, corner-rooted plane maps with the
root corner in the outer face. Hence the blossoming trees involved are
\emph{balanced}, meaning that no local closure may wrap the root
corner. More formally, let us define the (non-cyclic) \emph{contour
  word} of a rooted blossoming tree as the natural counterpart of the
cyclic contour word starting at the root corner, see
Definition~\ref{def:contourword}. Then a rooted blossoming tree is
said \emph{balanced} if the restriction of its contour word on
$\{b,\bar b\}$ is a Dyck word.

The intuition behind these bijections was originally relying on the
interpretation of the enumerative formulas. We want to emphasize here
that most of the time, a natural choice for a function $\alpha$ leads
to the same construction. We only sketch how to retrieve the proofs
that can be found in the original papers.

%%%%%%%%%%%%%%%%%%%%%%%%%%%%%%%%%%%%%%%%%%%%%%%%%%%%%%%%%%%%%%%%%%%%%
\subsection{Maps with prescribed vertex degree sequence}\label{sub:degseq}
%%%%%%%%%%%%%%%%%%%%%%%%%%%%%%%%%%%%%%%%%%%%%%%%%%%%%%%%%%%%%%%%%%%%%
  
%%%%%%%%%%%%%%%%%%%%%%%%%%%%%%%%%%%%%%%%%%%%%%%%%%%%%%%%%%%%%%%%%%%%%
\paragraph{Eulerian maps}
The first bijection obtained by Schaeffer in~\cite{Schaeffer97}
concerns planar \emph{Eulerian maps} with prescribed vertex degrees,
and in particular 4-regular planar maps with $n$ vertices, that
correspond bijectively to planar maps with $n$ edges.

This bijection can be recovered in the following way. First recall
that a map is said \emph{Eulerian} if its vertices have even degrees.
It is a classical result that Eulerian maps may be endowed with
orientations with equal in- and out-degrees for each vertex. Besisdes these 
orientations are accessible. In particular, the minimal \emph{Eulerian orientation}
of a given plane Eulerian map can be obtained recursively by orienting
clockwise the outer cycle and erasing it, see
Fig.~\ref{fig:oriEulerian}.

\begin{figure}
  \centering 
  \subfigure[\label{fig:oriEulerian}Its minimal Eulerian orientation,]{\includegraphics[page=3,scale=0.7]{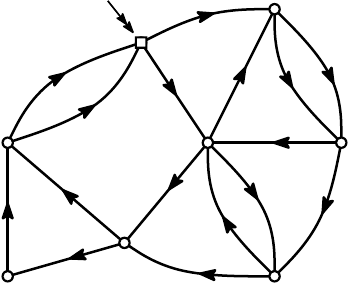}} \qquad
  \subfigure[its \TCP,]{\includegraphics[page=4,scale=0.7]{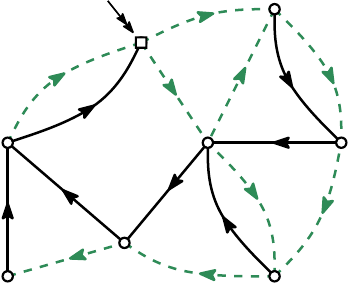}} \qquad
  \subfigure[and the corresponding balanced blossoming
  tree.]{\includegraphics[page=5,scale=0.7]{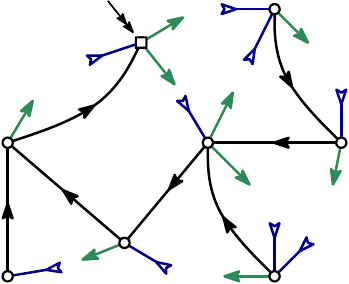}}
  
  \caption{Blossoming trees for an Eulerian map.}
  \label{fig:eulerian}
\end{figure}

The generic opening of a planar rooted Eulerian map $M$ with $n_i$
vertices of degree $2i$ for any $i \in \[[1,k\]]$ (and hence $n =
\sum_i i n_i$ edges) endowed with its minimal Eulerian orientation
leads to a balanced rooted blossoming tree with the same distribution
of in- and out-vertex degrees, and both $\ell = 1 + \sum_i(i-1)n_i$
opening and closing stems.  Observe that each non-root vertex has
exactly one outgoing edge that belongs to the blossoming tree. Moreover,
since the tree is balanced, the root corner is necessarily followed by
an opening stem.
\smallskip

To enumerate balanced blossoming trees, a general strategy is to
consider a larger family of planted blossoming trees that is stable by
rerooting and in which the proportion of balanced ones can be
evaluated. Let us sketch this strategy in the case of Eulerian maps,
following~\cite{Schaeffer97}.  Let us first consider planted trees with
$n_i$ nodes of degree $i+1$ (\ie arity $i$) for any $i>0$ and hence
$\ell + 1 = 2 + \sum_i(i-1)n_i$ leaves \emph{including the root one};
they are enumerated by (\cite{HaPrTu}):
\[
T_{n_1,\dots,n_k} ~=~ \frac1n \binom{n}{\ell,n_1,\dots,n_k} ~=~ \frac{(n-1)!}{\ell!}
\prod_{i=1}^k \frac1{n_i!}.
\]
 Start from one such tree and add $(i-1)$ opening stems on each node of arity
$i$. The total number of trees that can be obtained in this way is then: 
\[
B_{n_1,\dots,n_k} ~=~ \prod_{i=1}^k\binom{2i-1}{i}^{n_i} T_{n_1,\dots,n_k} ~=~
\frac{(n-1)!}{\ell!} \prod_{i=1}^k \binom{2i-1}{i}^{n_i}\frac1{n_i!}.
\]
Consider now each leaf (including the root one) as a closing stem, it
yields a blossoming tree with $\ell-1$ opening stems and $\ell+1$
closing stems, whose closure gives a map with two unmatched closing
stems.
\smallskip

Among all the planted blossoming trees that give the same map, a proportion
$2/(\ell+1)$ of them are rooted on one of the unmatched stems. For those ones,
changing their root (closing) stem into an opening one leads to a balanced
blossoming tree. Hence:

\bigbreak

\begin{prop}{\bf(Eulerian planar maps with prescribed vertex degrees)}
  The number of rooted planar Eulerian maps with $n_i$ vertices of
  degree $2i$ for any $i \in \[[1,k\]]$ is given by:
  \[
  \frac{2 \cdot (n-1)!}{(\ell+1)!} \prod_{i=1}^k \binom{2i-1}{i}^{n_i}\frac1{n_i!}.
  \]
\end{prop}
 
An interesting particular case concerns rooted planar 4-regular maps
with $n$ vertices, that correspond bijectively to rooted planar maps
with $n$ edges. Indeed, let $M$ be a planar map and consider its
radial map $\RR M$ (see Fig.\ref{fig:radial}).
By convention, $\RR M$ is rooted with the same root face as $M$, and
its root vertex corresponds to the root edge of $M$.
\smallbreak

It is clear from its definition that $\RR M$ is a 4-regular map and
that reciprocally every rooted 4-regular map with $n$ vertices
corresponds to a unique rooted planar map with $n$ edges.  According
to our generic bijective scheme, rooted 4-regular maps are in
bijection with balanced planted blossoming trees with $n$ nodes, of
in- and out-degrees 2, that is obtained from a planted binary tree by
adding one opening stem to each node. Hence:
\begin{cor}{\bf(Planar maps with prescribed number of edges)}\label{cor:generalmaps}
  The number of rooted planar maps with $n$ edges is:
  \[
  \frac{2 \cdot 3^n}{(n+2)(n+1)} \binom{2n}{n}.
  \]
\end{cor}

%%%%%%%%%%%%%%%%%%%%%%%%%%%%%%%%%%%%%%%%%%%%%%%%%%%%%%%%%%%%%%%%%%%%%
\paragraph{General maps}
In \cite{BDG02}, the bijection for Eulerian maps is generalized into a
bijection between planar maps with prescribed vertex degree sequence
and some blossoming trees. We sketch in this paragraph how this
construction can be derived from our generic scheme.

\smallbreak

\begin{figure}[t]
  \centering \subfigure[\label{fig:non-Eulerian-closure}The minimal quasi-Eulerian orientation,]{%
    \includegraphics[page=4,scale=0.7]{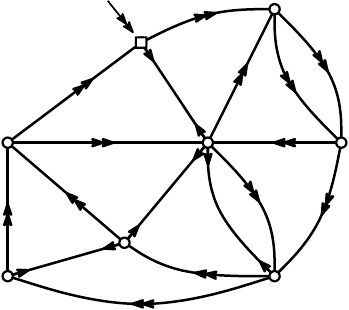}} \qquad
  \subfigure[the \TCP,]{\includegraphics[page=5,scale=0.7]{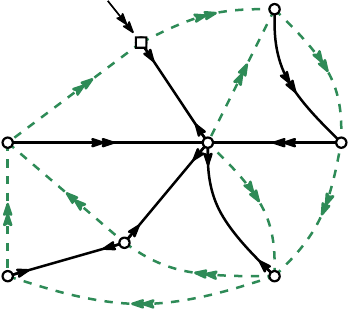}} \qquad
  \subfigure[\label{fig:non-Eulerian-blossom}and the corresponding balanced blossoming
  tree.]{\includegraphics[page=6,scale=0.7]{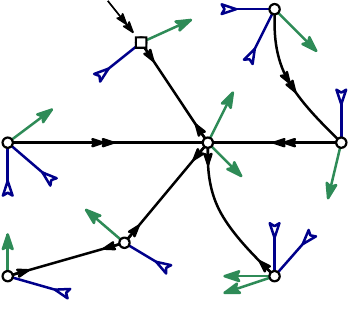}}
  \caption{Blossoming trees for general maps: the generic R-case.}
  \label{fig:non-Eulerian-generic}
\end{figure}

Any maps may be endowed with a \emph{quasi-Eulerian} orientation, that
is, a partial orientation (or a 2-fractional orientation) with equal
in- and out-degrees for each vertex. In particular, as minimal
Eulerian orientations, minimal quasi-Eulerian orientations can be
obtained recursively; orient clockwise the outer cycle (with the
convention that an edge adjacent twice to the outer face is partially
oriented in both directions) and iterate after erasing outer edges
(see Fig.~\ref{fig:non-Eulerian-closure} and ~\ref{fig:isthmus-closure}).

Let $M$ be a rooted map endowed with its minimal quasi-Eulerian
orientation. Opening $M$ gives a rooted balanced blossoming tree
endowed with an accessible $2$-fractional orientation such that the
in- and out-degrees of each vertex are equal. Now, to characterize
blossoming trees that admit such a $2$-fractional orientation, it is
convenient to follow~\cite{BDG02} and introduce the notion of
\emph{charge} of a subtree as the difference between the numbers of
its closing stems and opening stems. It is then easily seen that a
blossoming tree with total charge 0 admits an orientation with equal
in- and out-degrees at each vertex if and only if its proper subtrees
all have charge 0 or 1. More precisely, subtrees planted at a
saturated edge have charge 1, while those planted at a bi-oriented
edge have charge 0. This corresponds respectively to the R- and
S-trees in~\cite{BDG02}.
\smallskip

In particular, in the generic case where the root edge $e$ of $M$ is
not an isthmus, it is oriented clockwise and thus belongs to the
closure, therefore \BB{M} consists of an R-tree and an opening stem
situated right after the root corner (see
Fig.~\ref{fig:non-Eulerian-generic}). In the special case where $e$ is
an isthmus, it is partially oriented and belongs to the blossoming
tree $\BB M$. Therefore the right subtree of $\BB M$ is an S-tree
carried by $e$ and the other subtrees also form an S-tree (see
Fig.~\ref{fig:non-Eulerian-isthmus}).
\smallskip

Let us mention that the enumeration of these trees is not
straightforward (and is carried out in Section~3 of~\cite{BDG02}). The
main difficulty comes from the fact that R-trees are not stable by
rerooting.

\begin{figure}
  \centering \subfigure[\label{fig:isthmus-closure}The minimal quasi-Eulerian orientation,]{%
    \includegraphics[page=1,scale=0.7]{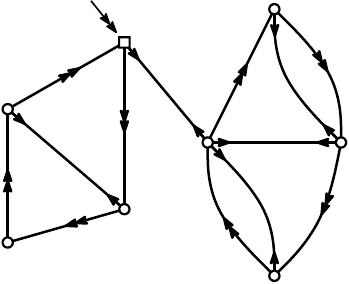}} \qquad
  \subfigure[the \TCP,]{\includegraphics[page=2,scale=0.7]{noneulerienne_isthme}} \qquad
  \subfigure[and the corresponding balanced blossoming
  tree.]{\includegraphics[page=3,scale=0.7]{noneulerienne_isthme}}
  \caption{Blossoming trees for general maps: the isthmus S-case.}
  \label{fig:non-Eulerian-isthmus}
\end{figure}

\bigbreak

%%%%%%%%%%%%%%%%%%%%%%%%%%%%%%%%%%%%%%%%%%%%%%%%%%%%%%%%%%%%%%%%%%%%%
\subsection{$m$-Eulerian maps with prescribed degree sequence}\label{sub:mEul}
%%%%%%%%%%%%%%%%%%%%%%%%%%%%%%%%%%%%%%%%%%%%%%%%%%%%%%%%%%%%%%%%%%%%%

In \cite{MBMSch00}, the authors define \emph{$m$-Eulerian} maps as
bipartite maps such that black vertices all have degree $m$ and each
white vertex has degree multiple of $m$. We give in this section a
much shorter proof that these maps are in bijection with the so-called
\emph{$m$-Eulerian trees}. Considering black vertices as hyper-edges,
$m$-Eulerian maps form in fact a subclass of $m$-regular hypermaps,
which boils down to Eulerian maps if $m=2$. Let us root $m$-Eulerian
maps in a white corner, and call the extremities of the root edge
respectively the \emph{white} and the \emph{black
  root vertex}.

An important feature of these maps is that they can be canonically
labelled on edges with integers in $\[[1,m\]]$ in such a way that:

\begin{itemize}
\item the root edge has label 1;
\item around each black vertex, edges are labelled from 1 to $m$ in
  clockwise order;
\item around each white vertex, edges are cyclically labelled in
  counterclockwise order; in particular, if $v$ is a white vertex with
  degree $km$, it is incident to exactly $k$ edges of label $i$
  for any $i$ in $\[[1,m\]]$.
\end{itemize}

\begin{figure}
  \centering
  \subfigure[The canonical orientation,\label{fig:const_canonic}]{\includegraphics[page=2,scale=0.65]{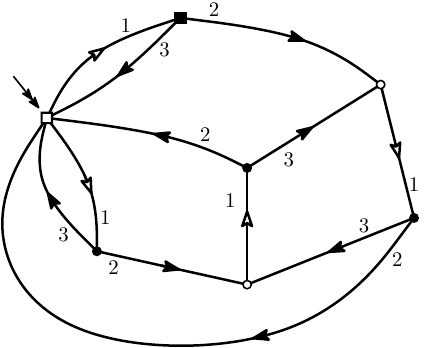}}\qquad
  \subfigure[the minimal one,\label{fig:const_min}]{\includegraphics[page=3,scale=0.65]{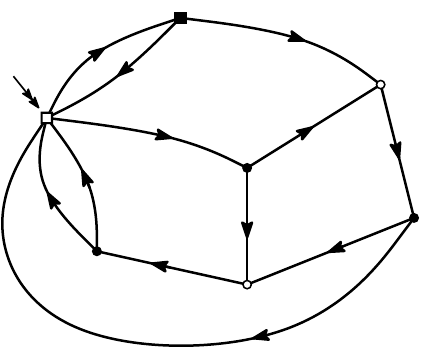}}\quad
  \subfigure[the \TCP,\label{fig:const_TCP}]{\quad\includegraphics[page=4,scale=0.65]{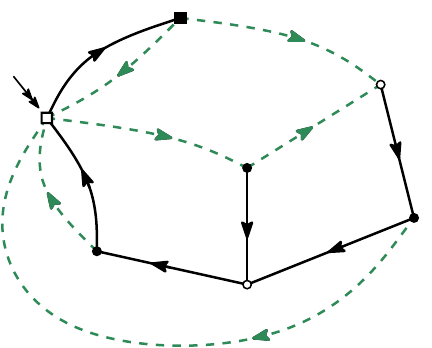}}\\
  \subfigure[the blossoming tree,\label{fig:const_blossom}]{\includegraphics[page=5,scale=0.65]{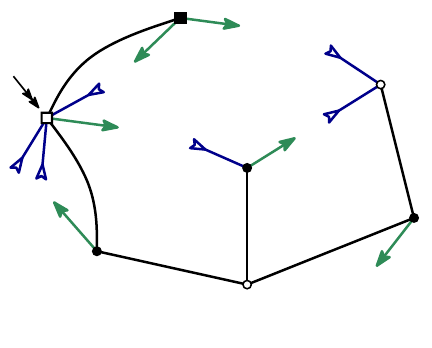}}\qquad
  \subfigure[the final swap,\label{fig:const_swap}]{\includegraphics[page=6,scale=0.65]{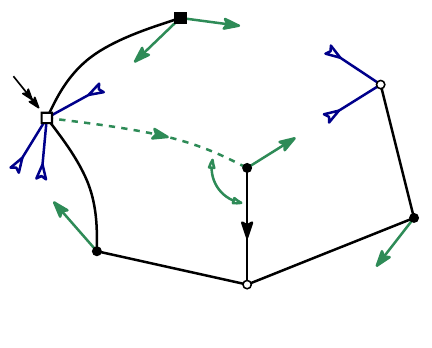}}\qquad
  \subfigure[and the 3-Eulerian tree.\label{fig:const_tree}]{\includegraphics[page=7,scale=0.65]{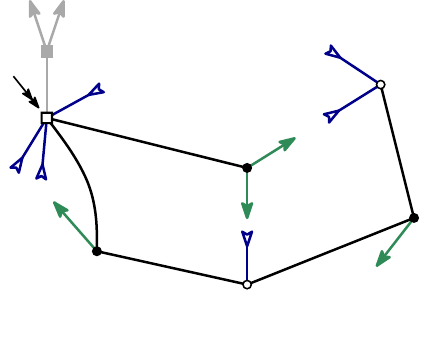}}
  \caption{Blossoming trees for $m$-Eulerian maps: example of a
    3-Eulerian map with two white vertices of degree 3 and one with
    degree 6, rooted on the edge with square ends.}
  \label{fig:constellation}
\end{figure}
 
\smallskip

This implies that any $m$-Eulerian map can be endowed with a
\emph{canonical} orientation: orient all edges but the 1-labelled ones
from their black end to their white end, see
Fig.~\ref{fig:const_canonic}. We will now prove that this canonical
orientation is accessible. First observe that, for any face $f$,
labels of edges incident to $f$ are alternatively equal to $i$ and
$i+1 \mod m$ for a given $i$. Hence if $i=1$ or $m$, $f$ is an
oriented cycle and all its vertices lie in the same strong connected
component. For other values of $i$, its edges are all oriented from
black to white, the $i$-labelled ones counterclockwise and the others
clockwise. Hence, if its vertices were not in one and the same
component, there would exist a cocycle which contains both an
$i$-labelled edge $e_0$ and an $i+1$-labelled edge $e_1$ incident to
$f$. Iterating this argument along the cocycle produces a cyclical
sequence of edges $(e_j)$ such that the label of $e_j$ is equal to
$i+j$. Since no 1- or $m$-labelled edge can belong to a cocycle, we
obtain a contradiction.
\bigbreak

Observe that in the canonical orientation of any $m$-Eulerian map,
each black vertex has outdegree equal to $m-1$ (and indegree equal to
1), and each white vertex of degree $km$ has outdegree equal to
$k$. Consequently, the function $\alpha$ defined by $\alpha(v) = m-1$
if $v$ is black, and $\alpha(v) = k$ if $v$ is white
with degree $km$, is accessibly feasible.

Now any rooted $m$-Eulerian map $M$ 
%with a given distribution of white vertex degrees 
can be canonically embedded in the plane with the root face as outer
face, and endowed with its minimal $\alpha$-orientation. The generic
opening according to the \emph{black} root vertex leads to a rooted
blossoming plane bipartite tree $T$ such that (see
Fig.~\ref{fig:const_blossom}):

\begin{itemize}
\item the (black) root vertex has only one child -- the white root
  vertex, and carries $m-1$ opening stems;
\item any white vertex with total degree $km$ carries $k-1$ opening stems, and
  $k(m-1)$ black children or closing stems;
\item black non-root vertices carry $m-2$ opening stems and either a white
  child or a closing stem.
\end{itemize}

$T$ is not exactly a \emph{$m$-Eulerian tree} as defined in
\cite{MBMSch00}, but the transition between the two families is easy.
It is enough to modify the tree locally in a way such that black and
white vertices respectively carry only opening and closing stems. To
do so, observe that opening stems carried by white vertices are
matched with closing stems carried by black ones (since the underlying
map is bipartite). Now suppose that such a couple is carried by a
white vertex $u$ and a black vertex $v$. It can be replaced by its
closure edge $(u,v)$, creating a cycle that is broken by opening the
edge connecting $v$ to its father in $T$ so as to create a new couple
of opening and closing stems (see
Fig.~\ref{fig:const_swap}-\subref{fig:const_tree}).  This \emph{swap}
leads to a blossoming tree with only opening stems on black vertices
and closing stems on white vertices; removing the black root vertex
gives precisely an $m$-Eulerian tree. As shown in~\cite{MBMSch00}:
\begin{prop}
  Let $m\geq 2$, the number of edge-rooted $m$-Eulerian maps with
  $d_i$ white vertices of degree $mi$ for each $i\geq 1$ is:
\[
m(m-1)^{v_\circ-1}\frac{[(m-1)v_\bullet]!}{[(m-1)v_\bullet-v_\circ+2]!}\prod_{i\geq
1}\frac{1}{d_i!}\binom{mi-1}{i-1}^{d_i},
\]
where $v_\bullet=\sum id_i$ and $v_\circ=\sum d_i$ denote respectively the
number of black and white vertices. 
\end{prop}

\subsection{Non-separable maps with prescribed number of edges}\label{sub:nonsep}

A \emph{cut vertex} of a map is a vertex that is incident twice to the
same face; a map is said to be \emph{non-separable} if it has no cut
vertex. Observe that this definition is stable by duality. It includes
the two maps with only one edge, but in the following, we will
consider only maps with at least two edges. In \cite{Sch98},
bijections are described for general and cubic non-separable planar
maps, and we show here that they are indeed special cases of our
generic bijective scheme.

\medskip

A \emph{bipolar orientation} of a map (or more generally a graph) is
an acyclic orientation of its edges with a single \emph{source}
(vertex without any incoming edge) and a single \emph{sink} (vertex
without any outgoing edge), which are called the \emph{poles} of the
orientation.  Non-separable maps (or graphs) are characterized by the
following property (see \eg\cite{FraPOMRos95}):

\begin{prop}\label{prop:nonsepbip}
  A rooted map is non-separable if and only if it can be endowed with
  a bipolar orientation with the two ends of the root edge as poles.
\end{prop}

\begin{figure}
  \centering
  \subfigure[Face configuration,\label{fig:prop_bipolar_face}]{\includegraphics[scale=1.2,page=1]{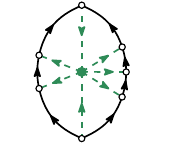}}\qquad
  \subfigure[vertex configuration,\label{fig:prop_bipolar_vertex}]{\includegraphics[scale=1.2,page=2]{nonseparable}}\qquad
  \subfigure[forbidden pattern in minimal bipolar orientations\label{fig:prop_bipolar_min}]{\quad\includegraphics[scale=1.2,page=3]{nonseparable}\quad}
  \caption{Properties of bipolar orientations; the non-separable map
    is drawn in plain lines, its quadrangulation in dashed green ones.}
  \label{fig:prop_bipolar}
\end{figure}

In the planar case, rooted non-separable maps endowed with a bipolar
orientation (with, say, the root vertex as the sink) have some
interesting properties, illustrated in Fig.~\ref{fig:prop_bipolar}:
\begin{enumerate}
\item each face is itself bipolar, hence its corners may be classified
  into \emph{lateral} ones (left or right) and two \emph{polar} ones
  (source and sink);
\item each vertex but the two poles has exactly one bundle of incoming
  edges and one bundle of outgoing ones, hence its corners may be
  classified into some polar ones (source or sink) and two lateral
  ones (left and right);
\item the quadrangulation $Q_M$ of the map can be
  endowed with an $\alpha$-orientation so that each extremity of the root edge has indegree 0 and every other vertex has 2 incoming edges (one for each special corner of the corresponding face or vertex of $M$, see Fig.\ref{fig:prop_bipolar_face} and~\ref{fig:prop_bipolar_vertex});
\item  reciprocally to each such orientation of $Q_M$ corresponds a bipolar orientation of $M$, hence 
  the set of bipolar orientations of a plane non-separable map $M$ is
  endowed with a lattice structure inherited from the lattice
  structure of the $\alpha$-orientations of its quadrangulation;
\item the minimal element of this lattice is the unique bipolar
  orientation of $M$ such that, for any distinct vertices $u$ and $v$
  both incident to distinct faces $f$ and $g$, the following
  configuration is forbidden: $u$ sink of $f$ and on the left of $g$,
  $v$ on the right of $f$ and source of $g$, see
  Fig.~\ref{fig:prop_bipolar_min}. This pattern corresponds to a
  counterclockwise 4-cycle in the quadrangulation.
\end{enumerate}

The constructions for general and cubic non-separable planar maps are
based on this minimal bipolar orientation. In the following, we denote
by $s$ and $t$ respectively the source and the sink of the considered
bipolar orientations, and say that a face is \emph{generic} if it is
\emph{not} incident to the root edge. We use the usual convention that
bipolar oriented maps are drawn in the plane with oriented paths going
upwards. Moreover $s$ and $t$ are outer vertices, $s$ is at the bottom and
$t$ at the top of the figure, and with this convention we choose the
embedding in which the root edge is the rightmost one.

\subsubsection{Non-separable cubic maps}

The case of cubic maps (treated in \cite{Sch98}, extended in
\cite{PouSch03}) is very constrained: since each vertex has degree 3,
in any bipolar orientation, each non-polar vertex has either one
incoming edge and two outgoing ones, or the opposite. Hence each
vertex (including the two poles) is either the source or the sink of exactly
one generic face. Let $M$ be such a rooted non-separable cubic map
endowed with its minimal bipolar orientation, and let us add an extra
\emph{bipolar edge} in each generic face of $M$ (that is, an edge
between its two poles). These extra edges realize a perfect matching
of the vertices, hence the resulting planar map $\overline M$ is
4-regular. As such, $\overline M$ can be endowed with its minimal
Eulerian orientation, and opened accordingly into a blossoming tree
rooted at the sink~$t$. We say that $\overline M$ is the
\emph{completion} of $M$.

\begin{lem}\label{lem:bipolar_edges}
  All bipolar edges belong to the resulting blossoming tree $\BB M$,
  and each vertex carries exactly one opening stem, immediately before
  the bipolar edge in clockwise order.
\end{lem}

\begin{proof}
  This is proved inductively on the number of non-polar vertices of
  $M$. The lemma is true in the smallest case (three parallel edges
  between $s$ and $t$). Let now $M$ have at least two inner vertices.
  Let $f$ be the (only) generic face incident to $s$, $g$ the bounded
  non-generic face, and $u$ the sink of~$f$. The vertex $u$ may be
  incident to the outer face (and even be equal to $t$); this specific
  case is illustrated in Fig.~\ref{fig:specific_cubic}, and the
  generic case is illustrated in Fig.~\ref{fig:generic_cubic}.  We
  show that the bipolar edge $(s,u)$ satisfies the lemma, and in each
  possible case we build smaller non-separable cubic maps such that
  $\BB{M}$ is obtained from their respective blossoming trees by
  grafting them together with the small subtree made of $s$, $u$ and
  their incident edges and stems.

\begin{figure}
  \centering
  \subfigure[The completion of $M$ and its Eulerian orientation,\label{fig:generic_cubic_a}]{\quad%
    \includegraphics[scale=1., page=1]{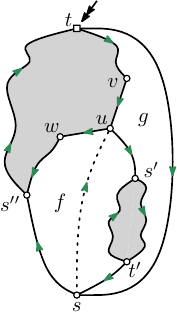}\quad}
  \subfigure[the tree-and-closure partition,]{\quad%
    \includegraphics[scale=1., page=2]{generique_cubique_nonseparable}\quad}
  \subfigure[the shape of the blossoming tree,]{\quad%
    \includegraphics[scale=1., page=3]{generique_cubique_nonseparable}}
  \subfigure[$\overline M'$ and~$\overline M''$.]{\qquad%
    \includegraphics[scale=1., page=4]{generique_cubique_nonseparable}\quad}
  \caption{Generic case of non-separable cubic maps.}
  \label{fig:generic_cubic}
\end{figure}

\medskip
\noindent\emph{Let us first prove that the bipolar edge $(s,u)$ satisfies the lemma.}
\smallskip

Properties of the minimal Eulerian orientation of
$\overline M$ imply that the border of the outer face is a clockwise
cycle, and that the edges incident to $g$ are oriented clockwise or
counterclockwise around $g$ depending on whether they are incident to
the outer face or not.  Observe also that the two outer edges incident
to $s$ are oriented clockwise and since they are unnecessary to the
accessibility of~$t$, they belong to the closure.  Therefore the
bipolar edge $(s,u)$, oriented from $s$ to $u$, belongs to the tree.

% To get that an opening stem precedes the bipolar edge around $u$,
% the key point is to see that, due to the minimality of the
% underlying bipolar orientation, 

Let us now prove by contradiction that the face on the right of $u$ is
the non-generic bounded face~$g$, that is, its source and sink are $s$
and~$t$.  Otherwise, let us define sequences of faces $(f_{i})$ and
$(g_{i})$ and of vertices $(s_{i})$, $(u_{i})$ and $(v_{i})$ as
follows. Let $f_{0}$ be equal to $f$, and for any $i\geq 0$, let
$s_{i}$ and $u_{i}$ be respectively the source and the sink of
$f_{i}$, $g_{i}$~the face on the right of $u_{i}$, $v_{i}$ its source,
and finally $f_{i+1}$ the face on the left of $v_{i}$. By assumption,
$g_{0}$ is different from $g$ (and clearly also from $f$), hence
$v_{0}$ is not equal to $s$, and $f_{1}$ is well defined. More
generally, it can be checked that $v_{i}$ cannot be equal to $s$ for
any $i\geq 1$, hence the whole sequence $(f_{i})$ is well defined.
  
Let \(\lambda_i\) be the leftmost path from \(s\) to \(s_i\), and
\(\rho_i\) be the rightmost path from \(s\) to \(v_i\). Denote now by
$R_{i}$ the region delimited by \(\lambda_i\) and the left border of
\(f_i\) on the left, and \(\rho_i\) and the left border of \(g_i\) on
the right. Then $u_{i+1}$ belongs to $R_{i}$, hence $(u_{i})$ is
decreasing for the partial order induced by the acyclic
orientation. It is therefore stationary after some~\(i_0\).  By
definition, $u_{i_0}$, $v_{i_0}$, $f_{i_0}$ and $g_{i_0}$ define
exactly the forbidden pattern~\ref{fig:prop_bipolar_min}, hence the
underlying bipolar orientation is not minimal, a contradiction.

% call $\tilde g$ the face on the right of $u$, $v$ its source, and
% let $\tilde f$ be the face on the left of $v$ and $\tilde u$ the
% sink of $\tilde f$ ($\tilde f$ and $\tilde u$ may be equal to $f$
% and $u$ respectively). Then $\tilde u$, $v$, $\tilde f$ and $\tilde
% g$ would define exactly the forbidden
% pattern~\ref{fig:prop_bipolar_min}, hence the underlying bipolar
% orientation would not be minimal, a contradiction.

Since the face on the right of $u$ is $g$, the edge following the
bipolar edge $(s,u)$ in counterclockwise order around $u$ is incident
to both $f$ and $g$. Hence it is necessarily outgoing, and belongs to
the closure, which proves the lemma locally for $s$ and~$u$.

\medskip
\noindent\emph{Let us decompose $M$ into smaller non-separable cubic maps.}
\smallskip

\begin{figure}
  \centering
  \subfigure[The completion of $M$ and its Eulerian orientation,\label{fig:specific_cubic_a}]{\quad%
    \includegraphics[scale=1., page=1]{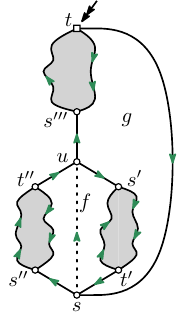}\quad}
  \subfigure[the tree-and-closure partition,]{\quad%
    \includegraphics[scale=1., page=2]{serie_cubique_nonseparable}\quad}
  \subfigure[the shape of the blossoming tree,]{\quad%
    \includegraphics[scale=1., page=3]{serie_cubique_nonseparable}}
  \subfigure[$\overline M'$, $\overline M''$ and~$\overline M'''$.]{\qquad%
    \includegraphics[scale=1., page=4]{serie_cubique_nonseparable}\quad}
  \caption{Specific \emph{series parallel} case for non-separable cubic
    maps.}
  \label{fig:specific_cubic}
\end{figure}

Now let $s'$ and $t'$ be the respective neighbours of $u$ and $s$ that
are both incident to $f$ and $g$ -- possibly equal to $s$ and $u$, see
Fig. \ref{fig:generic_cubic}-\ref{fig:specific_cubic}(a). If $s'\neq
s$ and $t'\neq u$, deleting $(u, s')$ and $(t',s)$ disconnects $M$
into two submaps; as these edges are both oriented clockwise around
$f$, $(u,s')$ belongs to the closure and $(t', s)$ belongs to the
tree, see
Fig. \ref{fig:generic_cubic}-\ref{fig:specific_cubic}(b). Adding a
root edge between $s'$ and $t'$ with root vertex $t'$ in the
corresponding submap leads to a smaller non-separable cubic map~$M'$,
see Fig. \ref{fig:generic_cubic}-\ref{fig:specific_cubic}(d). Observe
that the restriction to $M'$ of the minimal bipolar orientation of $M$
is precisely the opposite of the minimal bipolar orientation of
$M'$. Hence these two orientations define the same matching between
the sources and the sinks of generic faces. Thus the completion of
$M'$ is a submap of $\overline M$ (except for its root edge), and its
minimal Eulerian orientation is the restriction of the one of
$\overline M$.

We resume to $s$ and $u$ being possibly equal to $s'$ and $t'$.
To end the construction in the generic case, let $s''$ be the last
neighbour of $s$, and let $v$ and $w$ be the two last neighbours of
$u$, with $v$ incident to $g$ and $w$ incident to $f$, see
Fig.~\ref{fig:generic_cubic}. Then $(u,w)$ is a tree edge, while
$(s,s'')$ and $(v,u)$ are closure edges. Let $M''$ be obtained from
$M$ by removing $s$ and $u$ and their incident edges, and adding to
the component of $t$ and $s''$ successively in the outer face a
(closure) edge between $v$ and $w$, and a (closure) root edge between
$t$ and $s''$. Then $M''$ is a smaller non-separable cubic map, whose
minimal bipolar orientation is the restriction of the minimal bipolar
orientation of $M$.  Hence $\BB{M}$ is obtained from $\BB{M''}$ by
grafting the subtree made from $u$, $s$ and $\BB{M'}$ instead of the
suitable stem of~$w$.

Now, in the case where $u$ is incident to the outer face, the
situation between the face $f$ and the outer face is similar to the
one between the faces $f$ and $g$, as illustrated in
Fig.~\ref{fig:specific_cubic}. Let $s''$ and $t''$ be the neighbours
of $s$ and $u$ between these two faces, then $(s,s'')$ belongs to the
closure and, as soon as the two edges are distinct, $(t'',t)$ belongs
to the tree.  Adding a root edge between $s''$ and $t''$ in the
corresponding submap leads to a smaller non-separable cubic map~$M''$.

If moreover $u\neq t$, let $s'''$ be its fourth neighbour; the edge
$(u,s''')$ (oriented towards $s'''$) is incident to the two
non-generic faces, hence deleting the root edge and $(u,s''')$
disconnects $M$, which implies in particular that $(u,s''')$ belongs
to the tree. In this case, let $M'''$ denote the submap containing
$s'''$ and $t$, with an additional root edge between $s'''$ and~$t$:
$M$ is somehow a \emph{series parallel} compound of three submaps
$M'$, $M''$ and $M'''$, each possibly empty in degenerate cases, see
Fig.~\ref{fig:specific_cubic}.
\end{proof}

%It is then quite clear that these balanced blossoming trees are
%exactly those of \cite{Sch98, PouSch03}, and can be described in a
%very simple manner. 

\begin{figure}
  \centering
  \subfigure[The minimal bipolar
  orientation of~$M$,]{\;\includegraphics[scale=1, page=1]{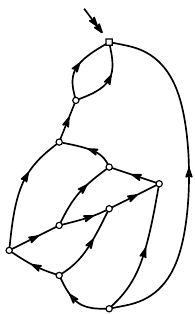}\quad}\quad
 %% \subfigure[the two types of non-polar vertices,]{\includegraphics[scale=1, page=1]{cubique_nonseparable}}\;
  \subfigure[$\overline M$ with the extra bipolar edges,]{\includegraphics[scale=1, page=2]{exemple_cubique_nonseparable}}\qquad
  \subfigure[its minimal Eulerian orientation,]{\includegraphics[scale=1, page=3]{exemple_cubique_nonseparable}}\qquad
  \subfigure[and the resulting blossoming tree.]{\includegraphics[scale=1, page=5]{exemple_cubique_nonseparable}}
  \caption{Complete example of a non-separable cubic map $M$.}
  \label{fig:nonsep_cubic}
\end{figure}

Following \cite{Sch98, PouSch03}, let us define a blossoming
\emph{twin ternary tree} as a tree obtained by fairly splitting each
node of a ternary tree into two \emph{twin} nodes linked by a special
\emph{twinning edge}, with an additional opening stem on each node
right before its twinning edge in clockwise order (see Fig.\ref{fig:nonsep_cubic}(d)).  

Observe that the opening of the completion of a non-separable cubic
map leads to a (balanced) blossoming twin ternary tree, where twinning edges are bipolar edges. Indeed, non-root opening stems always follow bipolar
edges in counterclockwise order around vertices and the contraction of bipolar edges yields a planted ternary tree (the leaves of which are the closing stems). 

Conversely, we prove now that the closure of any balanced blossoming
twin ternary tree is the completion of a non-separable cubic map.  For
any subtree (a node, its stems and its descendants), let us call
\emph{free} the stems that are matched with stems not belonging to the
subtree.  An immediate counting shows that each subtree has one more
closing stem than opening ones, hence one more free closing stem than
free opening ones.  Let us denote by $s$ the node carrying the closing
stem $c$ corresponding to the (opening) root stem, and let $u$ be its
twin node; let us a prove by contradiction that $u$ is its parent. If $u$ were the
right child of $s$, its opening stem would be matched with the (only)
closing stem of $s$. If $u$ were its left child, the subtree of $s$
would have at least two free closing stems since the opening stem of
$u$ would indeed be free, which prevents $c$ from being matched with
the root stem. Hence $u$ is necessarily the parent of $s$, and as the
opening stem of $u$ is matched to a closing stem in the right subtree
of $s$, $c$ is necessarily just before the opening stem of $s$ in
clockwise order around~$s$. Hence $s$ and $u$ are exactly in the
configuration of Fig.\ref{fig:generic_cubic_a} or Fig.\ref{fig:specific_cubic_a}.
We can conclude by a similar induction argument as in the proof of Lemma~\ref{lem:bipolar_edges}.

Hence rooted non-separable cubic planar maps with $2n$ vertices are in
one-to-one correspondence with balanced blossoming twin ternary trees
with $2n$ nodes. The number of planted ternary trees with $n$ nodes is
$\frac1{2n+1}\binom{3n}{n}$, each one leading to $2^n$ distinct
planted blossoming twin ternary trees. A fraction $\frac2{2n+2}$ of
them is balanced, leading to:

\begin{cor}
  The number of rooted planar non-separable cubic maps with $2n$
  vertices and $3n$ edges is equal to:
\[
\frac{2^{n}}{(n+1)(2n+1)} \binom{3n}{n}.
\]
\end{cor}

\subsubsection{General non-separable maps}

In the case of general planar non-separable maps, the generic scheme
is applied on (an extension of) their radial maps, as in the proof of
Corollary~\ref{cor:generalmaps}. Let $M$ be a rooted non-separable
map, and $\RR M$ its radial map, rooted with the same root face as
$M$, and root vertex $r$ corresponding to the root edge of $M$.  The
orientation of the quadrangulation of $M$ defined below
Proposition~\ref{prop:nonsepbip} yields naturally an orientation of
$\RR M$ such that each generic face of $\RR M$ has exactly two
clockwise edges (corresponding to the two special corners of the
corresponding face or vertex of $M$). For edges incident to $r$, we
adopt the convention that the two outer ones are oriented clockwise,
and the last two ones are outgoing for the root vertex. This
orientation is said \emph{minimal} if the corresponding bipolar
orientation of $M$ is itself minimal, see
Fig.~\ref{fig:nonsep_minimale_radiale}.

  \begin{figure}
    \centering
    \subfigure[The minimal orientations of $M$ and $\RR M$,\label{fig:nonsep_minimale_radiale}]{\includegraphics[scale=0.85, page=2]{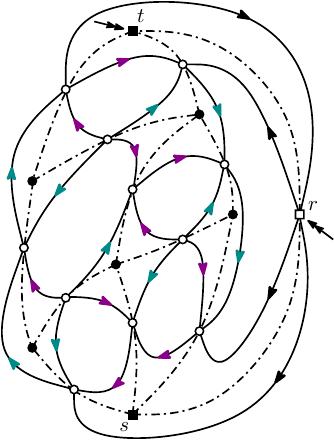}}\qquad
    \subfigure[the augmented map $\overline{\RR M}$,\label{fig:non_sep_extra}]{\includegraphics[scale=.85, page=4]{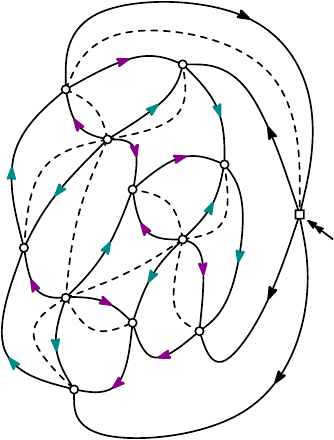}}\qquad
    \subfigure[and the blossoming tree.\label{fig:non_sep_arbre}]{\includegraphics[scale=.85, page=5]{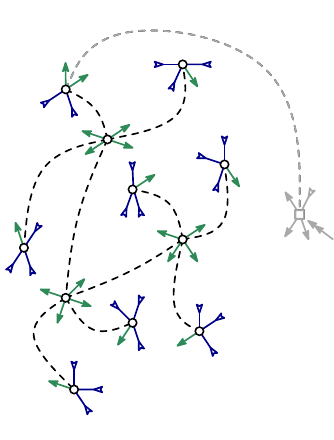}\!\!\!\!}
    \caption{Complete example of a non-separable planar map $M$.}
    \label{fig:nonsep_exemple}
  \end{figure}

Given such an orientation of $\RR M$, each generic face has two special corners
(the origins of the two clockwise edges), and so does the face that
corresponds to $t$. Let us add an extra edge in each such face between
these two corners, and denote respectively $\overline{\RR M}$ the resulting
map and $\TT M$ the map made exactly of these extra edges and their
incident vertices, see Fig.~\ref{fig:non_sep_extra}. Then:

\begin{lem}[\cite{Sch98}]\label{lem:spanning}
  $\TT M$ is a spanning tree of $\overline{\RR M}$ if and only if the underlying
  orientation of $\RR M$ is minimal. Moreover in this case, the edges of $\RR M$ go clockwise around $\TT M$. 
\end{lem}

Hence in the minimal case, $\RR M$ is actually a valid set of closure
edges around $\TT M$, whatever accessible orientation is chosen for
edges in $\TT M$. For instance, we may orient all edges towards $r$,
or simply leave them unoriented (that is, oriented both ways in a
2-fractional orientation).  Then $(\TT M, \RR M)$ is the \TCP of
$\overline{\RR M}$.

The resulting (balanced) blossoming tree is such that each non-root
vertex is incident to 4 stems, and to as many edges as opening stems --
hence it has in-degree equal to 4. Considering closing stems as
leaves, and after some surgery to remove the root vertex, we get a
planted ternary tree with one extra opening stem at each corner before
an inner edge (clockwise around each vertex), see
Fig.~\ref{fig:non_sep_arbre}. Reciprocally, as shown in
\cite{Sch98}, the closure edges of such a balanced blossoming tree $T$ form a
4-regular map $R$ endowed with an orientation with 2 clockwise edges
per face,
which by Lemma~\ref{lem:spanning} ensures that $R$ is the radial map
of a non-separable map, endowed with its minimal orientation. Since any
planted blossoming ternary tree with \(n\) nodes has \(2n+2\) leaves (including
the root stem) and \(2n-2\) opening stems, a fraction \(\frac{4}{2n+2}\)
among them is balanced,  hence:

\begin{cor}
  For $n\geq 1$, the number of rooted planar non-separable maps with $n+1$
  edges is equal~to:
\[
\frac{2}{(n+1)(2n+1)} \binom{3n}{n}.
\]
\end{cor}

\subsection{Simple triangulations and quadrangulations}
Bijections between simple triangulations or quadrangulations and
blossoming trees, as described in \cite{PouSch06, FusyThesis}, are
special cases of the general bijection for $d$-angulations of girth
$d$ as explained in Sections~\ref{sec:dangulations}
and~\ref{sec:opening}, with a special emphasis in
Subsection~\ref{sub:trigquad}. In particular, the uniqueness part in
Theorem~\ref{thm:open} gives a more direct proof that the closure
construction of~\cite{PouSch06,FusyThesis} for simple triangulations
and quadrangulations of a $p$-gon is injective, while the existence
part proves surjectivity without requiring a cardinality argument.

Besides, it is noteworthy that in such cases where the degree of faces
are prescribed, closing stems are redundant; since the underlying
orientation is regular, blossoming trees are trees with a fixed number
of opening stems per vertex.

%%%%%%%%%%%%%%%%%%%%%%%%%%%%%%%%%%%%%%%%%%%%%%%%%%%%%%%%%%%%%%%%%%%%%
\section{Plane bipolar orientations}\label{sec:bipolar}
%%%%%%%%%%%%%%%%%%%%%%%%%%%%%%%%%%%%%%%%%%%%%%%%%%%%%%%%%%%%%%%%%%%%%

Recall that a bipolar orientation of a map is an acyclic orientation
of its edges with a single source (vertex without incoming edge) and a
single sink (vertex without outgoing edge).  A \emph{plane bipolar
  orientation} is a (non-separable) corner-rooted map (with at least two edges) endowed with a bipolar
orientation such that the root vertex of the map is the sink of the
orientation and its source is the other extremity of the root edge,
see Fig.~\ref{fig:ex_bipolar}. We emphasize that this section is
devoted to the study of \emph{all} plane bipolar orientations, as
opposed to Section~\ref{sub:nonsep} which focuses only on maps endowed
with their minimal bipolar orientation as a tool to enumerate non-separable maps.
Plane bipolar orientations have a nice enumerative formula:

\begin{thm}[Baxter~\cite{Baxter01}]
  For all non-negative integers $i$ and $j$, the number $\Theta_{ij}$ of
  plane bipolar orientations with $i$ non-pole vertices and $j$
  generic faces is equal to:
  \begin{equation}
    \Theta_{ij} = \frac{2\;(i+j)!\;(i+j+1)!\;(i+j+2)!}{ i!\;(i + 1)!\; (i +
      2)!\ \ j!\;(j + 1)!\;(j + 2)!}.
    \label{eq:nb-bipolar}
  \end{equation}
\end{thm}
The first proof of this formula was given by Baxter~\cite{Baxter01}.
His proof involves quite technical computation and relies on a ``guess
and check'' approach.  Since then, some bijective proofs of this
result have been obtained in~\cite{FuPoSc09},~\cite{BoBoFu10}
and~\cite{FeFuNoOr11}. Our generic scheme provides a new bijective
proof.

\subsection{Bijection with triple of paths}\label{sub:defbip}

Let us consider paths on $\ZZ^2$ made of right-steps $(0,1)$ and
up-steps $(1,0)$. A configuration of such paths is called
\emph{non-intersecting} if each vertex of $\ZZ^2$ belongs to at most
one path. For $i,j\in \NN$, define the set $\Pbip_{i,j}$ of
non-intersecting triple of paths $(\p1,\p2,\p3)$, each made of exactly
$i$ right- and $j$ up-steps and starting respectively at $(-1,1)$,
$(0,0)$ and $(1,-1)$ (and hence ending at $(i-1,j+1)$, $(i,j)$ and
$(i+1,j-1)$), see Fig.~\ref{fig:paths_bipolar}.

The rest of this section is devoted to the proof of the following
theorem, from which a direct application of Lindström-Gessel-Viennot
Lemma \cite{GesVie89} yields the enumerative result of Baxter:
\begin{thm}\label{thm:bij_bip}
  For all positive integers $i$ and $j$, there exists a one-to-one
  constructive correspondence between the set of plane bipolar
  orientations with $i$ generic faces and $j$ non-pole vertices and
  the set $\Pbip_{i,j}$.
\end{thm}

Other bijections between plane bipolar orientations and the set
$\Pbip$ already appear in~\cite{FuPoSc09,FeFuNoOr11}. It must
nevertheless be emphasized that the bijection we obtained is
different, thus providing a first example where our general scheme
yields a new bijective construction.  To prove the theorem, we start
by applying the generic scheme to open a bipolar orientation into a
blossoming tree, which is then encoded by a non-intersecting triple of
paths.
\begin{figure}[t]
  \centering
\captionsetup{justification=centering}
  \subfigure[A plane bipolar
  orientation,\label{fig:ex_bipolar}]{\qquad\quad%
    \includegraphics[scale=1, page=1]{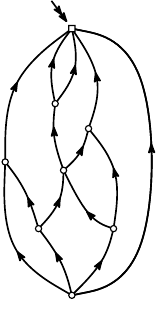}\qquad\quad}\quad%
  \subfigure[its blossoming tree,\label{fig:tree_bipolar}]{\qquad%
    \includegraphics[scale=1,page=3]{orientation_bipolaire}\qquad}\qquad%
 \subfigure[the resulting triple of
  paths.\label{fig:paths_bipolar}]{%
    \includegraphics[scale=0.9, page=1]{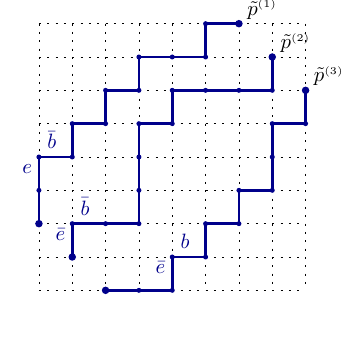}}
  \caption{Example of the bijection for a bipolar orientation $M$.\\ The contour word $w$ of its blossoming tree is $ebbb\be ee \bb b \be e \bb e \bb b \be b \be \be e \bb b \be \bb e \bb \bb \be$.\\
Then $\w1=e ee \bb e \bb e \bb e \bb \bb e \bb \bb$, $\w2= \be \bb \be \bb \bb \be \be \be \bb \be \bb \bb \bb \be$ and 
$\w3= bbb \be b \be b \be b\be \be b \be \be$. }
 \label{fig:bipolar}
\end{figure}

%%%%%%%%%%%%%%%%%%%%%%%%%%%%%%%%%%%%%%%%%%%%%%%%%%%%%%%%%%%%%%%%%%%%%
\subsection{From bipolar orientations to configurations of
  paths}\label{sec:treetopaths}
%%%%%%%%%%%%%%%%%%%%%%%%%%%%%%%%%%%%%%%%%%%%%%%%%%%%%%%%%%%%%%%%%%%%% 
Since any bipolar orientation is acyclic and its sink vertex $t$ is
accessible from any vertex, Corollary~\ref{cor:blossom} can be applied
to open it into a blossoming tree, rooted at the former outer corner
of $t$. 
Let $\Tbip$ (resp. $\Tbip_{i,j}$) be the set of balanced blossoming
trees obtained when opening a plane bipolar orientation (resp. with
$i$ generic faces and $j$ non-pole vertices).

Let $T$ be a rooted blossoming tree, we consider its contour word $w$
on the alphabet $\{e,\be, b, \bb\}$, where $e$ and $\be$ encode
respectively the first and second exploration of an edge and $b$ and
$\bb$ encode respectively the exploration of an opening or a closing
stem. For instance the contour word of the tree of
Fig.\ref{fig:tree_bipolar} is $ebbb\be ee \bb b \be e \bb e \bb b \be
b \be \be e \bb b \be \bb e \bb \bb \be$. Let furthermore define $\w
1$, $\w2$ and $\w3$ as the subwords of $w$ obtained respectively by
keeping only the letters $e$ and $\bb$, $\be$ and $\bb$, $\be$ and
$b$:
\[
\w1 = w_{|e,\bb},\quad \w2=w_{|\be,\bb}\quad\text{and}\quad \w3=w_{|\be,b}.
\]

\begin{claim}\label{claim:trip-word}
  For $T$ a blossoming tree in $\Tbip_{i,j}$, the words $\w1$, $\w2$ and $\w3$
  have length $\ell=i+j+2$ and furthermore: 
  \[\w1_1=e\text{ and }\,\w1_{\ell}=\bb,\qquad\quad
  \w2_1=\be \text{ and }\,\w2_2=\bb,\qquad\quad
  \w3_1=b\text{ and }\,\w3_{\ell}=\be,\]
  where $w_k$ denotes the $k$-th letter of a word $w$. 
\end{claim}
This claim is a consequence of Lemma~\ref{lem:Tbip} below, we omit its proof.

\noindent For $T\in \Tbip$, Claim~\ref{claim:trip-word} enables to define $(\tw1,\tw2,\tw3)$ as: 
\begin{equation}
  \w1
  = e\,\tw1\,\bb,\quad \w2= \be\bb\,\tw2\quad\text{and}\quad \w3 = b
  \,\tw3\, \be. 
  \label{eq:w-tilde}
\end{equation}
Triple of words $(\tw1,\tw2,\tw3)$ can be naturally represented by
triple of up-right paths $(\tp1,\tp2,\tp3)$, with initial points
$(-1,1)$, $(0,0)$ and $(1,-1)$, by replacing letters $e$ or $\be$ by
up-steps and letters $b$ and $\bb$ by right-steps. Let $\Phi$ be the
application that associates to each tree of $\Tbip$ the corresponding
triple of paths, see Fig.~\ref{fig:bipolar}.  Observe that if $M$ has
$i$ generic faces and $j$ non-pole vertices, then the corresponding
blossoming tree $T$ has $i+1$ pairs of opening-closing stems and $j+1$
edges, therefore each of the paths $\tp1$, $\tp2$ and $\tp3$ have
exactly $i$ right-steps and $j$ up-steps.

\begin{prop}\label{prop:treetopaths}
Let $T$ be an element of $\Tbip_{i,j}$. Then its image by $\Phi$ is non-intersecting, in other words
it belongs to $\Pbip_{i,j}$. 
\end{prop}

\noindent This follows from:

\begin{lem}\label{lem:Tbip}
  A word $w$ on the alphabet $\{e,\be,b,\bb\}$ is the contour word of
  an element of $\Tbip$ if and only if the five following conditions
  hold:

  \begin{itemize}
  \item[(1)] $w_1=e$ and $w_2=b$;
    \qquad(2) $w_{|e,\be}$ is a Dyck word;
    \qquad(3) $w_{|b,\bb}$ is a Dyck word;
  \item[(4)] for all $i<j$, if $w_i=b$ and $w_j=\bb$, 
    there exists $k\in [i,j]$ such that $w_k=\be$;
  \item[(5)] for all $1<i<j$, if $w_i=e$ and $w_j=\be$, there exists
    $k\in[i,j]$ such that $w_k=\bb$. 
  \end{itemize}
\end{lem}

\begin{proof}
  Let $T$ be an element of $\Tbip$ and let $w$ be its contour
  word. For the first condition, observe that the root edge, oriented
  from $s$ to $t$, belongs to the blossoming tree and is the first
  edge encountered in the contour of $T$, thus giving a first letter
  $e$. Moreover, $s$ is a leaf of $T$ that carries only opening stems
  and at least one, so $w_2=b$. Condition~(2) reflects the fact that $T$
  is a tree and Condition~(3) that it is balanced.

  Conditions~(4) and~(5) are proved by contradiction. If Condition~(4) does
  not hold, there exists a factor of $w$ of the form $be\cdots e\bb$
  or $b\bb$. It implies that a vertex is matched with one of its
  descendants (possibly itself) in the closure, producing an oriented
  cycle, a contradiction.  Similarly, if Condition~(5) does not hold,
  there exists a leaf of $T$ (different from $s$) which does not carry
  a closing stem. It contradicts the uniqueness of the source.
  \smallskip

  Reciprocally, Conditions~(2) and~(3) imply that $w$ is the
  contour word of a balanced blossoming tree $T$. Moreover, Conditions~(1), (4) and~(5) imply that 
the first subtree of the root is reduced to one edge $(t,s)$, where $s$
  carries at least one opening stem. 
 To ensure that $T$ belongs to $\Tbip$, it is enough to prove (a) that each
  vertex of $T$ different from $s$ or $t$ has at least one ingoing and one outgoing edges and
  (b) that the orientation of the closure is acyclic. 
  
 Each node of $T$ has at least one ingoing edge per child and Condition~(4) ensures that each leaf of $T$ carries at least one closing stem. Moreover each vertex but $t$ has one outgoing edge (oriented towards its parent in $T$), hence (a) is satisfied. 
 
If the opening stem $b$ is matched with the closing stem $\bar b$, then the corner incident to $b$ is explored before the one incident to $\bar b$ in the contour process of $T$. It implies that if there exists an oriented cycle in the closure of $T$, at least one closure edge links a node to one of its descendants in $T$. Let us prove that this cannot happen, since in the contour word of $T$, each occurrence of $b$ is followed by $b^\star \bar e$. Consider an occurrence of $b$ in $w$ and the first occurrence of $\bb$ after it. Consider the first occurrence of $\be$ after $b$ (which precedes $\bb$ by Condition~(4)).  The corresponding occurrence of $e$ necessarily precedes $b$ by Condition~(5), hence there is no occurrences of $e$ between $b$ and $\be$. Hence $(b)$ is satisfied.
\end{proof}

\begin{proof}[Proof of Proposition~\ref{prop:treetopaths}]
  For any $\ell \geq 0$, denote $w_{\sss{\leq \ell}}$ the prefix of length $\ell$ of a word $w$,
  and $|w|_x$ the number of occurrences of $x$ in $w$. Let $w$ be the contour word of $T$, Condition~(2) implies that:
  \[
  \forall \ell\geq 0, \quad |\w1_{\sss{\leq \ell}}|_{\bb}\leq|\w2_{\sss{\leq \ell}}|_{\bb}\quad \text{and} \quad
  |\w1_{\sss{\leq \ell}}|_{e}\geq|\w2_{\sss{\leq \ell}}|_{\be}.
  \]
  Consequently, the corresponding paths $(\p1,\p2)$ starting at
  $(0,0)$ are such that $\p1$ lies above and on the left of $\p2$ with
  possible common vertices or edges. However the two paths share no
  vertical edge but the first one by Condition~(5). Hence, after
  translating $\p1$ by an up-step, $\p1$ and $\p2$ are
  non-intersecting and so are $\tp1$ and $\tp2$.  \smallskip

  We can prove similarly that $(\tp2,\tp3)$ is non-intersecting by
  observing first that $\p2$ lies above and on the left of $\p3$
  thanks to Condition~(3) of Lemma~\ref{lem:Tbip} and that $\p2$ and
  $\p3$ cannot share horizontal edge by Condition~(4). Translating $\p3$
  by a right-step and deleting the appropriate steps of $\p2$ and
  $\p3$ yield the desired result.
\end{proof}

\subsection{From configurations of paths to blossoming
trees}\label{sec:pathstotree}
Let $p=(\tp1, \tp2, \tp3)$ be a configuration of paths in
$\Pbip_{i,j}$ and $(\w1,\w2,\w3)$ the corresponding triple of
words. Let us decompose $\w1$ and $\w3$ as a sequence of factors
according respectively to the occurrences of $\bb$ and $\be$:
\[
\w1 = e\cdot \w1_{\sss{[1]}}\cdot \w1_{\sss{[2]}}\cdot \ldots\cdot
\w1_{\sss{[j+1]}}\qquad \text{ and }\qquad\w3 =b\cdot
\w3_{\sss{[1]}}\cdot \w3_{\sss{[2]}}\cdot \ldots\cdot
\w3_{\sss{[i+1]}},
\]
where each factor $\w1_{\sss{[k]}}$ (resp. $\w3_{\sss{[k]}}$) is of
the form $e^\star\bb$ (resp. $b^{\star}\be$).  These factors are
``bricks'' used to reconstruct a \emph{compatible} word $w$, that is a word such that: 
\[
\w1 = w_{|e,\bb},\quad \w2=w_{|\be,\bb}\quad\text{and}\quad \w3=w_{|\be,b}.
\]
The order in
which those bricks are added is driven by $\w2$: let $\bar w$ be
obtained from $\w2$ by replacing its $k$-th occurrence of $\bb$ by
$\w1_{\sss{[k]}}$ and its $k$-th occurrence of $\be$ by
$\w3_{\sss{[k]}}$, and define finally $w = eb\bar w$.

\begin{prop}\label{prop:pathstotree}
  Let $p$ be an element of $\Pbip_{i,j}$ and let $w$ be the
  corresponding word as defined above. The word $w$ is the unique word
  compatible with $p$ that satisfies the five conditions of
  Lemma~\ref{lem:Tbip}. In other words it is the contour word of the
  unique blossoming tree of $\Tbip$ compatible with $p$.
\end{prop}

\begin{proof}
  First observe that $w$ is compatible with $p$ and that no other such
  word may satisfy the conditions of Lemma~\ref{lem:Tbip}:
  Conditions~(4) and~(5) imply that the factors $\w1_{\sss{[k]}}$ and
  $\w3_{\sss{[\ell]}}$ have to be factors of $w$, and the order in
  which they appear is completely determined by $\w2$.

  Let us now prove that $w$ encodes indeed an element of $\Tbip$, by
  applying Lemma~\ref{lem:Tbip}.  Condition~(1) is clearly
  satisfied. Conditions~(4) and~(5) follow also easily from the definition
  of the decomposition in factors of $\w1$ and $\w3$: observe for
  instance, for Condition~(4), that any occurrence of $b$ in $w$ (but
  the first one) comes from a factor $\w3_{\sss{[k]}}$ of the form
  $b^{\star}\be$.  The first occurrence of $b$ does not either raise a
  problem since $\w2_{1}=\be$ and is hence replaced by
  $\w3_{\sss{[1]}}$.

  It remains to prove Conditions~(2) and~(3), namely that $w_{|e,\be}$ and
  $w_{|b,\bb}$ are Dyck words. We only give the proof for
  $w_{|e,\be}$, since both proofs work along the same lines.  From the
  construction of $w$, the number of occurrences of $e$ and of $\be$
  in $w$ are both equal to $i+1$, hence it is enough to prove that
  $|w_{\sss\leq k}|_{e}\geq |w_{\sss\leq k}|_{\be}$, for all $k$.  We
  consider the following decomposition of $w$ into product of factors:
  \[
  w = eb\cdot w_{\sss{[1]}}\cdot w_{\sss{[2]}}\cdot \ldots\cdot
  w_{\sss{[i+j+2]}},
  \]
  where each of the $w_{\sss{[k]}}$ is equal to the corresponding
  factor of $\w1$ or of $\w3$. It is then enough to check that for
  each $1\leq k \leq i+j+2$:
  \[
  \Big| e b \prod_{i=1}^k w_{\sss{[i]}}\Big|_{e} \geq \Big| eb
  \prod_{i=1}^kw_{\sss{[i]}}\Big|_{\be}, 
  \]
  which can be rewritten as: 
  \[
  1+ \sum_{i=1}^{k_1}|\w1_{\sss{[i]}}|_{e} \geq \sum_{i=1}^k|\w2_{i}|_{\be},
  \quad \text{ where } k_1 = |\w2_{\sss{\leq k}}|_{\bb}.
  \]
  Let $(x_2,y_2)$ be the point of $\p2$ reached after $k$ steps and
  let $(x_2,y_1)$ be the point of $\p1$ of abscissa $x_2$ with minimal
  ordinate. By construction of $w$, the value of $y_1$ is equal to the
  left-hand side of the above inequality, while $y_2$ is equal to the
  right hand-side. Since $\p1$ lies above $\p2$, we obtain the desired
  result.
\end{proof}
%%%%%%%%%%%%%%%%%%%%%%%%%%%%%%%%%%%%%%%%%%%%%%%%%%%%%%%%%%%%%%%%%%%%%

%%%%%%%%%%%%%%%%%%%%%%%%%%%%%%%%%%%%%%%%%%%%%%%%%%%%%%%%%%%%%%%%%%%%%
\section{Blossoming trees for $d$-angulations}\label{sec:dangulations}
%%%%%%%%%%%%%%%%%%%%%%%%%%%%%%%%%%%%%%%%%%%%%%%%%%%%%%%%%%%%%%%%%%%%%
The aim of this section is to generalize bijections previously
obtained for simple triangulations~\cite{PouSch06} and simple
quadrangulations~\cite{FusyThesis}, that is triangulations and
quadrangulations without loops nor multiple edges. In other words,
triangulations and quadrangulations in which the contours of the faces
are shortest cycles. More generally, the \emph{girth} of a map is
defined as the minimal length of its cycles. Obviously a
$d$-angulation has girth at most $d$ (except if it is a tree), hence
simple triangulations and simple quadrangulations are exactly
triangulations and quadrangulations with maximal girth. In the
remaining sections, we aim at applying the general scheme to
$d$-angulations of girth $d$, for any $d\geq 3$, and also to their
following generalization. For any integers $p\geq d\geq 3$, define a
\emph{$d$-angulation of a $p$-gon} or a \emph{$p$-gonal
  $d$-angulation} as a face-rooted plane map such that the contour of
the root face is a simple cycle of length $p$ and all non-root faces
have degree $d$, see Fig.~\ref{fig:pgonal}. We denote respectively
$\Mdd$ and $\Mdp$ the set of $d$-angulations and $p$-gonal
$d$-angulations of girth $d$, with distinct root and outer faces.

We do not use here the canonical plane embedding of face-rooted maps
with the root face as the outer face. On the contrary, from now on, we
consider only \emph{face-rooted plane maps in which the outer face and
  root face are different}. This convention yields obviously
equivalent enumerative byproducts but proves to fit better.

\begin{figure}[t]
  \centering
  \subfigure[A simple quadrangulation endowed with its minimal
  2-orientation.]{%
    \quad\includegraphics[scale=0.9, page=1]{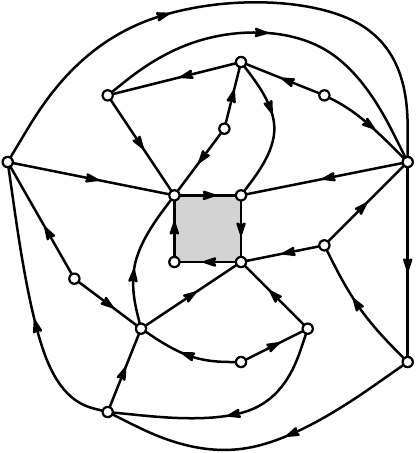}\quad}
  \qquad\qquad
  \subfigure[A pentagulation of girth $5$ endowed with its minimal
  $5/3$-orientation.\label{fig:oriPent}]{%
    \includegraphics[scale=0.95, page=1]{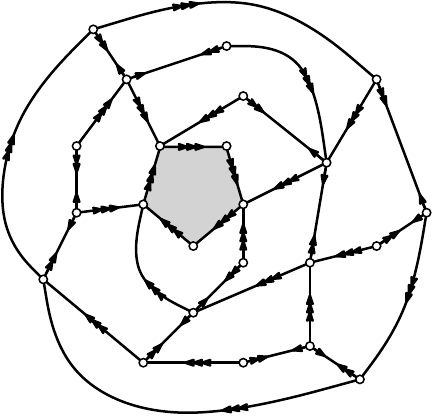}}
  \caption{Examples of $\frac{d}{d-2}$ orientations for
    $d$-angulations of girth~$d$.}
  \label{fig:exorientation}
\end{figure}

%%%%%%%%%%%%%%%%%%%%%%%%%%%%%%%%%%%%%%%%%%%%%%%%%%%%%%%%%%%%%%%%%%%%%
\subsection{Orientations for $p$-gonal $d$-angulations}
%%%%%%%%%%%%%%%%%%%%%%%%%%%%%%%%%%%%%%%%%%%%%%%%%%%%%%%%%%%%%%%%%%%%%
For any $j,k\geq 0$, a $j/k$-orientation of a face-rooted map is
defined as a $k$-fractional orientation such that for each root vertex
$v$, $\out(v)=k$, and $\out(v)=j$ otherwise (see
Fig.~\ref{fig:oriPent}). Bernardi and Fusy show
in~\cite{BerFusPentagulation} that the existence of \ddori{s}
characterizes $d$-angulations of girth $d$, generalizing previous
results obtained for triangulations \cite{Schnyder} and
quadrangulations \cite{OssonaThesis}:

\begin{thm}[Schnyder~\cite{Schnyder}, Ossona de
  Mendez~\cite{OssonaThesis}, Bernardi and
  Fusy~\cite{BerFusPentagulation}]
  \label{thm:girthd}
  Let $d\geq 3$ and $M$ be a face-rooted $d$-angulation; then $M$
  admits a \ddori if and only if it has girth $d$.  Moreover any such
  orientation is accessible.

  Besides if $d$ is even, all the flows in this orientation are even.
\end{thm}

\begin{rem}
  For $d$ even, the parity of the flows implies that maps of $\Mdd$
  can be endowed with $\frac{d/2}{d/2-1}$-orientations: for instance
  quadrangulations are naturally endowed with 2-orientations.  For
  sake of conciseness we work here mainly with \ddori{s} and
  distinguish odd and even cases only when needed.
\end{rem}

For $p>d$, a simple application of Euler formula proves that a
$p$-gonal $d$-angulation cannot admit a \ddori.  The appropriate
generalization is to define a \emph{(pseudo-) \ddori} as a
$(d-2)$-fractional orientation in which the contour of the root face
is a circuit of saturated edges and $\out(v)=d$ for any non-root
vertex $v$.  Observe that for $p=d$, minimal pseudo-\ddori{s} are
minimal \ddori{s}.  Pseudo-\ddori{s} characterize $p$-gonal
$d$-angulations of girth $d$:
\smallskip

\def\toto{\cite[Lemma~18]{BerFusPentagulation}}
\begin{prop}[\toto]
  Let $p\geq d\geq 3$ be integers. A $p$-gonal $d$-angulation $M$
  admits a pseudo-\ddori if and only if it has girth $d$. In this case
  every pseudo-\ddori is accessible, and the sum of the outdegrees of
  the root vertices is equal to $(d-2)p+(p-d)$.

  Besides, there exists a unique minimal such pseudo-\ddori.
\end{prop}
\smallskip

\begin{rem} Since we allow accessibility and minimality to be defined
  relatively to two different faces (the root face and the outer
  face), this minimal pseudo-\ddori fits in our
  scheme. In~\cite{BerFusPentagulation} however, accessibility and
  minimality have to be defined relatively to the same face. It
  implies \cite[Proposition 19]{BerFusPentagulation} that the
  existence of a minimal suitable orientation is conditioned on the
  map being non-separated (\ie there cannot exist a cycle of girth
  length that separates the root face and the outer face).
\end{rem}
\smallskip

\begin{figure}[t]
  \centering
  \subfigure[A $7$-gonal $5$-angulation of girth 5 endowed with its minimal
  pseudo $\frac{5}{3}$-orientation,\label{fig:pgonal}]{%
    \;\includegraphics[scale=1, page=1]{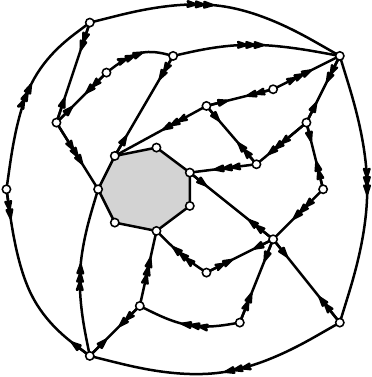}\;}
  \qquad\qquad
  \subfigure[and the corresponding $7$-gonal $5$-fractional
  forest.\label{fig:cyclicforest}]{%
    \includegraphics[scale=1, page=3]{7-pentagulation}}
  \caption{Example of the correspondence between $p$-gonal $d$-angulations of
    girth $d$ and $p$-gonal $d$-fractional forests.}
\end{figure}

%%%%%%%%%%%%%%%%%%%%%%%%%%%%%%%%%%%%%%%%%%%%%%%%%%%%%%%%%%%%%%%%%%%%%
\subsection{Bijection for $p$-gonal $d$-angulations}
%%%%%%%%%%%%%%%%%%%%%%%%%%%%%%%%%%%%%%%%%%%%%%%%%%%%%%%%%%%%%%%%%%%%%

To adapt Theorem~\ref{thm:open} to our context, we need to introduce
the following family of planar maps akin to forests of trees. A
\emph{$p$-cyclic forest} is a face-rooted plane map with two faces,
the root one and the outer one, such that the border of the root face
is a simple cycle of length $p$. Observe that a cyclic forest is
nothing else but a cycle of sequences of planted trees.

%\newpage

\begin{cor}\label{cor:open}
  Let $M=(V,E)$ be a plane face-rooted map with distinct root and
  outer faces and endowed with a minimal accessible orientation
  $O$. Then $M$ admits a unique edge-partition $(\TT M,\CC M)$ such
  that:

  \begin{itemize}
  \item edges in $\TT M$ form a spanning cyclic forest of $M$ with same
    root-face as $M$, on which the restriction of $O$ is accessible;
  \item edges in $\CC M$ are saturated edges, and any of them turns
    clockwise around the unique cycle it forms with edges in $\TT M$.
  \end{itemize}
\end{cor}

\begin{proof}
  The map $M$ admits such a unique partition of its edges if and
  only if $\tilde{M}$ does, where $\tilde{M}$ is the vertex-rooted map
  constructed from $M$ by contracting its root face. We conclude by applying Theorem~\ref{thm:open} to $\tilde M$.
\end{proof}  

The two following definitions describe the corresponding blossoming
objects in the setting of $p$-gonal $d$-angulations (see
Fig.~\ref{fig:pgonal}).

For any integers $d\geq 3$ and $0\leq i <d-2$, a \emph{$d$-fractional
  tree of excess $i$} is a planted blossoming tree endowed with an
accessible $(d-2)$-fractional orientation such that the root leaf has
outdegree $i$ and each non-root vertex has outdegree $\out(u)=d$,
where each opening stem contributes $d-2$. The set of $d$-fractional
trees of excess $i$ (resp. with $n$ vertices) is denoted by $\Tdn{i}$
(resp. $\Tdn i_n $).

A \emph{$p$-gonal $d$-fractional forest} is a $p$-cyclic forest, the
planted trees of which are $d$-fractional trees. The sum of their
excesses is moreover required to be equal to $p-d$. Observe that such
a forest is naturally endowed with a pseudo-\ddori.  The set of
$p$-gonal $d$-fractional forests (resp. with $n$ vertices) is denoted
by $\Fdp$ (resp. $\Fdp(n)$). The closure of a $p$-gonal $d$-fractional
forest is the natural counterpart of the closure of a blossoming map, in
which the local closure of an opening stem creates a face of
degree~$d$.

The main theorem of this section is the following application of
Corollary~\ref{cor:blossom} to $p$-gonal $d$-angulations.

\begin{thm}\label{thm:pseudo}
  There exists a one-to-one constructive correspondence between
  $p$-gonal $d$-fractional forests with $n$ vertices and $p$-gonal
  $d$-angulations of girth $d$ with $n$ non-root vertices.
\end{thm}
 
\begin{proof}
  Corollary~\ref{cor:open} and~\ref{cor:blossom} entail that a
  $p$-gonal $d$-angulation endowed with its minimal \ddori can
  uniquely be opened into a $p$-cyclic $d$-fractional forest with
  additional closing stems. Deleting them yields a $p$-cyclic
  $d$-fractional forest, from which they can be retrieved using the
  condition that non-root faces have degree $d$.

  The only point to check is that the closure of a $p$-gonal
  $d$-fractional forest is indeed a $p$-gonal $d$-angulation, \ie that
  the degree of the outer face of the full closure is~$d$.  A
  $d$-fractional tree of excess $i$ with $k$ vertices has
  $\frac{2k+i}{d-2}$ opening stems. So a $p$-gonal $d$-fractional
  forest with $n+p$ vertices has $\frac{2n+p-d}{d-2}$ opening stems
  and its outer face has degree $2n+p$. Since each local closure
  reduces this degree by $d-2$, it is exactly equal to $d$ at the end
  of the closing process.

  As mentioned in Section~\ref{sec:effective}, the complexity of the
  opening procedure in the general case is quadratic. For $p$-gonal
  $d$-angulations, we were able to find a linear constructive
  algorithm. Its description is postponed to Section~\ref{sec:opening}
  for sake of clarity.
\end{proof}

\subsection{Enumerative consequences}\label{sec:computation}

By Theorem~\ref{thm:pseudo}, the enumeration of $d$-angulations of
girth $d$ boils down to the enumeration of $d$-fractional forests and
the latter reduces to the counting of $d$-fractional trees of excess
$i$ for $i\in\{0,\ldots,d-3\}$. Let $\tdi i$ be the generating series
of $\Tdn i$ according to the number of opening stems. The only way to
generate and enumerate the elements of $\Tdn i _ n$ seems to require a
recursive approach, described in this section. It comes as no surprise
that the recursive scheme we obtain is essentially the same as the one
that counts the $d$-regular mobiles of \cite{BerFusPentagulation}. We
found it nevertheless interesting to note that the same enumerative
results can be obtained by our approach.

Following \cite{BerFusPentagulation}, for any positive integer $j$,
let $h_j$ be the polynomial in the variables $t_1,t_2,\ldots$ defined
by:
\[
h_j(t_1,t_2,\ldots) :=
[x^j]\frac{1}{1-\sum_{i>0}x^it_i} = \sum_{r>0}
\sum_{\substack{i_1,\ldots,i_r>0\\i_1+\ldots+i_r=j}} t_{i_1} \ldots t_{i_r}.
\] 
In other terms, $h_j$ is the generating function of compositions of
$j$ where the variable $t_i$ keeps track of the number of parts of
size $i$.  Now, the generating series of the fractional trees of
excess $i$ can be obtained by a recursive decomposition. Cut the
child of the root leaf to obtain a forest of trees (where one of the
tree can be reduced to an opening stem) such that the sum of their
excesses is equal to $i+2$. In other words it can be written as a
sequence $(s_0,t_1,s_1,\ldots,t_l,s_l)$, where $s_0,\ldots,s_l$ are
some sequences of trees with excess 0 (possibly reduced to a single
vertex) and each of the $t_i$'s is a tree with positive excess. It
yields the following system of equations:
\begin{equation}\label{eq:schemarecursif}
  \tdi i (x) = \frac{1}{1-\tdi 0}\cdot
    h_{i+2}\big(\frac{\tdi 1}{1-\tdi 0},\ldots,\frac{\tdi
    {d-3}}{1-\tdi {0}},\frac{x}{1-\tdi {0}}\big),
\end{equation}
for $0\leq i \leq d-3$ and $\tdi i = 0$ otherwise.  This set of
equations characterizes $\tdi 0, \tdi 1,\ldots, \tdi {d-3}$ as formal
power series. The constant coefficient of all
these series is clearly equal to zero and the other coefficients can
be computed recursively.

\begin{prop}[Bernardi and Fusy~\cite{BerFusPentagulation}]
  \label{prop:enubord}
  For $p\geq d\geq 3$, the generating function $M_{d,p}(x)$ of
  corner-rooted $p$-gonal $d$-angulations of girth $d$ with a marked
  outer face and counted according to the number of non-root faces is
  equal to:
  \[
  M_{d,p}(x)=x\left(\frac{1}{1-T_0}\right)^p\cdot\hdp {p-d}
  {p}\left(\frac{T_1}{1-T_0},\ldots,\frac{T_{d-3}}{1-T_0},\frac{x}{1-T_0}\right),
  \]
  where $\hdp j p$ is defined in~\eref{eq:hdp} and the series
  $T_0,\ldots,T_{d-3}$ are characterized by~\eref{eq:schemarecursif}.

For $p=d$, this equation reduces to: 
\[
  M_{d,d}(x)=x\left(\frac{1}{1-T_0}\right)^d
\]
\end{prop}

\begin{proof}
By Theorem~\ref{thm:pseudo}, we have:
\[
M_{d,p}(x)=xF_{d,p}(x),
\]
where $F_{d,p}$ is the generating function of corner-rooted $p$-gonal $d$-fractional
cyclic forests counted according to their number of opening stems.
Erasing the edges of the root face of such a forest produces a $p$-tuple
 $P:=(P_1,\ldots,P_p)$ of sequences of
fractional trees such that the total sum of their excesses is equal to
$p-d$.

For each $1\leq i \leq p$,  the
sequence $P_i$ can be written as
$P_i=(s_{i_0},t_{i_1},s_{i_1},\ldots,t_{i_l},s_{i_l})$, such that each
of the $s_j$ is a sequence made of trees of excess 0 (possibly empty)
and each $t_j$ is a tree of positive excess. This decomposition gives
the following formula for $F_{d,p}(x)$:
\[
F_{d,p}(x)=\left(\frac{1}{1-T_0}\right)^p\cdot\hdp {p-d}
{p}\left(\frac{T_1}{1-T_0},\ldots,\frac{T_{d-3}}{1-T_0},\frac{x}{1-T_0}\right),
\]
where for $j\geq 0$, the polynomial $\hdp j p$ is the generating
function of $p$-tuples of compositions of integers, the sum of which is
equal to $j$. For $p=0$, $\hdp j 0 =1$, for $p=1$ we retrieve the quantity $h_j$ and more
generally:
\begin{equation}\label{eq:hdp}
  \hdp j p (t_1,t_2,\ldots) :=
       [x^j]\frac{1}{(1-\sum_{i>0}x^it_i)^p}.
\vspace*{-.5ex}\end{equation}
\end{proof}

%%%%%%%%%%%%%%%%%%%%%%%%%%%%%%%%%%%%%%%%%%%%%%%%%%%%%%%%%%%%%%%%%%%%%
\section{Fast opening of $p$-gonal $d$-angulations of girth $d$}\label{sec:opening}
%%%%%%%%%%%%%%%%%%%%%%%%%%%%%%%%%%%%%%%%%%%%%%%%%%%%%%%%%%%%%%%%%%%%%

\subsection{General description of the algorithm}\label{sub:general}
This section is devoted to the description of a linear-time algorithm
that associates to each element of $\Mdd$ its tree-and-closure
partition or equivalently its $d$-fractional forest. We only give the
proof in this setting, all steps extend
easily to $p$-gonal $d$-angulations.

\medskip

Let $M \in \Mdd$; recall that \TTe{M}, \CC{M} and \BB{M} denote
respectively its sets of tree edges, of closure edges and its
$d$-fractional forest.  Since the root face of $M$ is not its outer
face, the construction of Proposition~\ref{prop:bernardi} cannot be
applied directly to obtain~\BB{M}. The general idea of our algorithm
is however to use a similar construction iteratively to identify and
cut subtrees of \BB{M}.

\begin{figure}[t]
  \centering \subfigure[A blossoming forest with a distinguished subtree $\mathcal{B}_{M}^{(u,v)}$. \label{sub fig:subtreeM}]{\includegraphics[page=3, scale =0.8]{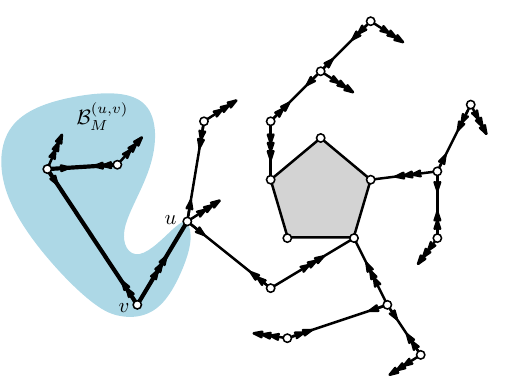}} \qquad
\subfigure[The almost-total closure $M(u,v)$ of $M$ relatively to $(u,v)$.\label{subfig:almost-total}]{\includegraphics[page=4, scale=0.8]{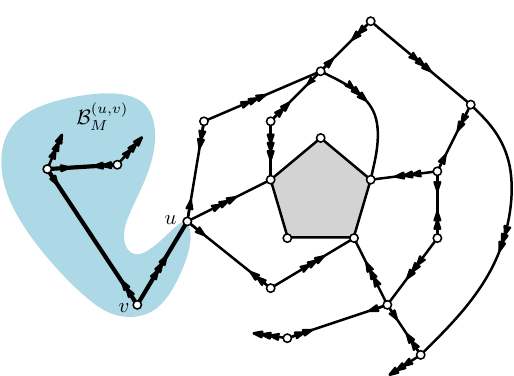}}  
\caption{Almost-total closure.\label{fig:almost-total}}
\end{figure}

\medskip

More precisely, for any edge $e\in \TTe M$, let us denote by \BB[e]{M}
the (blossoming) subtree of \BB{M} planted at~$e$. Define the
\emph{almost-total closure} \MM{e} of $M$ relatively to $e$ as the
maximal partial closure of \BB{M} such that the border of the subtree
\BB[e]{M} (including both sides of $e$) still lies on the outer face
of \MM{e} (see Fig.~\ref{subfig:almost-total}).  Deleting \BB[e]{M} yields
a blossoming map with a canonically marked corner, denoted by~\Mr{e}.
It is clear that $M$ is the closure of the blossoming map
obtained by grafting \BB[e]{M} on \Mr{e} in its marked corner. Then,

\medskip

\begin{claim}\label{claim:uni}
  For any map $M\in\Mdd$ and any edge $e\in\TTe{M}$, \BB{M} is the
  forest obtained by grafting \BB[e]{M} in the marked corner of~\BB{\Mr{e}}.
\end{claim}

\begin{proof} 
  It follows directly from the uniqueness of $\BB{M}$ (proved in
  Theorem~\ref{thm:pseudo}).
\end{proof}

\medskip

Claim~\ref{claim:uni} entails that computing the blossoming forest of
$M$ can be achieved in three steps: first, identify \MM{e} for some
tree-edge~$e$, then, compute recursively the blossoming forest
of~\Mr{e}, and last, graft \BB[e]{M} in the marked corner
of~\BB{\Mr{e}}. This is precisely described by
Algorithms~\ref{algo:lineage} to~\ref{algo:complet} page
\pageref{algo:complet}.
The correction of Algorithm~\ref{algo:complet} relies on the three
following propositions, whose proofs are postponed to the next
subsections:

\begin{algorithm}[p]
  \caption{LineagePath($M$)}\label{algo:lineage}
  \KwIn{a map $M\in \Mdd$ endowed with its minimal \ddori} 
  \KwOut{a lineage path $(v_0,e_1,v_1,\ldots,e_\ell,v_\ell)$ in $\TT M$ oriented
    towards the root face}  
  \KwSide{partial opening, such that each $e_i$ is an outer edge}
  \BlankLine
  let $(u,v)$ be an outer closure edge; initialize $C$ to $\{(u,v)\}$,
  $i$ to 0 and set $v_0 = v$\;   
  \Repeat{$i = \lceil d/2 \rceil$ or $v$ is a root vertex}{%
   update $u$ to the previous vertex around $v$ in clockwise order\;
    \lIf{$(u,v)$ is saturated towards $v$}{add it to $C$}
    \lElse{update $(u,v)$ to $(v,u)$, increment $i$ and set $v_i=v$
      and $e_{i}= (u,v)$}
  }
  replace each edge of $C$ by an opening stem in the appropriate corner\;
  \Return{$(v_0,e_1,v_1,\ldots,e_\ell,v_\ell)$}
\end{algorithm}

\begin{algorithm}[p]
\caption{DetachableSubtree($M$)}\label{algo:subtree}
 \KwIn{a map $M\in \Mdd$, endowed with its minimal \ddori}
 \KwOut{the set $T$ of inner edges of a subtree $\BB[e_{\ell}]{M}$}
\KwSide{M is partially opened into \MM{e_\ell}}
 \BlankLine
let $(v_0,e_1,v_1,\dots, e_{\ell},v_\ell) =
\FuncSty{LineagePath}(M)$\;
initialize $T$ and $C$ to $\emptyset$, and $(u,v)$ to
$(v_{\ell}, v_{\ell-1})$\;
 \Repeat{$(u,v) = (v_{\ell-1}, v_{\ell})$}{%
\If{$(u,v) \not\in C\cup T$}{%
     \lIf{$(u,v)$ is saturated towards $v$}{add $(u,v)$ to $C$}
     \lElse{add $(u,v)$ to $T$}
 }
 \lIf{$(u,v) \in T$}{%
     update $(u, v)$ to $(v, u)$
 }
 update $v$ to the next vertex (in clockwise order) around $u$\;
     }
 replace each edge of $C$ by an opening stem in the appropriate corner\;
\Return{$T$}
\end{algorithm}

\begin{algorithm}[p]
\caption{BlossomingForest(M)}\label{algo:complet}
  \KwIn{a map $M\in \Mdd$, endowed with its minimal \ddori}
  \KwOut{its blossoming forest \BB M}
  \BlankLine
  % 
  %$(e_1,\dots, e_{\ell}) \gets $ \FuncSty{LineagePath}($M$)\;
  %$T \gets $ \FuncSty{Subtree}($M, e_{\ell}$)\;
  let $T = \FuncSty{DetachableSubtree}(M)$ and $t$ be the minimal tree
  with same excess \;
  replace $T$ by $t$ in $M$ and perform the closure\;
  let $F = \FuncSty{BlossomingForest}(M)$  and replace $t$ by $T$ in $F$\;
  \Return{F}\;
\end{algorithm}

\begin{figure}[p]
  \centering \subfigure[Identification of a lineage path of
  length~3,\label{subfig:lineage}]{%
    \includegraphics[scale = 0.8,page=4]{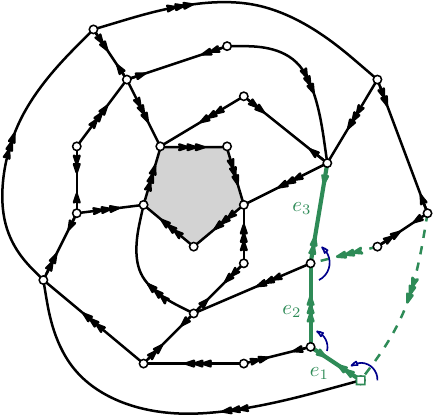}} 
 \quad\subfigure[the subtree ${\BB[e_3]{M}}$,\label{subfig:dfs}]{%
    \quad\includegraphics[scale=0.8,page=5]{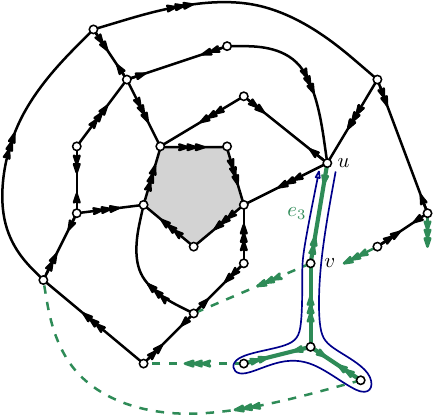}}

  \subfigure[replacing ${\BB[e_3]{M}}$ by $t_1^{(5)}$,\label{subfig:grafting}]{%
    \includegraphics[scale = 0.8,page=7]{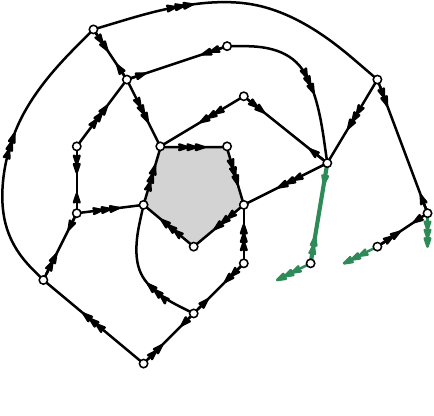}
    \mbox{}\vspace{-5em}} 
  \quad\subfigure[the lineage path reaches the root pentagon, 
  \label{subfig:lineage2}]{%
    \quad\includegraphics[scale=0.85,page=9]{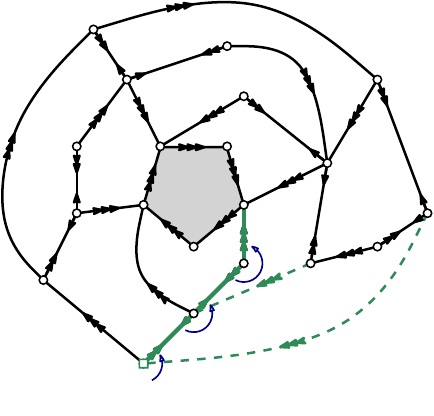}
    \mbox{}\vspace{-5em}}

  \subfigure[according to Remark~\ref{rem:fin}, the DFS identifies $\BB{\Mr{e_3}}$ entirely,\label{subfig:dfs2}]{%
    \includegraphics[scale = 0.8,page=10]{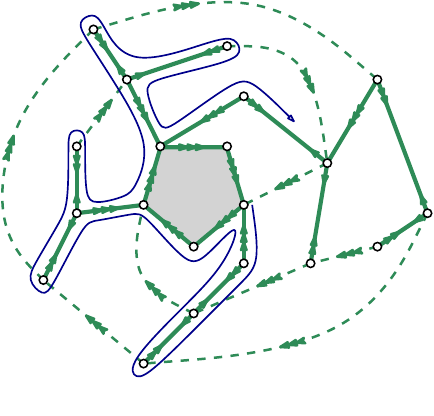}} 
 \quad\subfigure[which enables to reconstruct \BB{M}.\label{subfig:reconstruct}]{%
    \quad\includegraphics[scale=0.8,page=3]{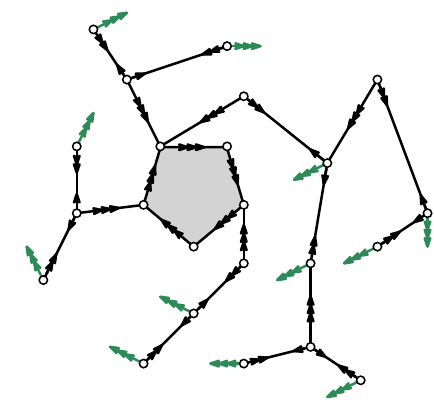}}
 
\caption{Execution of the fast opening algorithm on the
  pentagulation of Fig.\ref{fig:oriPent}.}
\label{fig:aretebord}
\end{figure}

\smallskip

\begin{prop}\label{lem:path} 
  Algorithm \ref{algo:lineage} returns a lineage path
  $(v_0,e_1,v_1,\ldots,e_\ell, v_\ell)$ in $\TT M$ oriented towards
  the root polygon and opens
  a subset $C$ of $\CC M$ such that each $e_i$ is an outer edge of
  the resulting blossoming map.
\end{prop}

\begin{prop}\label{prop:bernardi++} 
    Algorithm \ref{algo:subtree} returns the subtree $\BB[e]{M}$
    for some edge $e$ and opens the corresponding closure edges of $M$
    so as to compute~\MM{e}.
\end{prop}

\begin{prop}\label{prop:order}
  There exists a partial order on blossoming trees such that:

  \begin{itemize}
  \item for any $i \in \[[0, d-2\]]$, there exists a unique minimum $\td i$ among trees
    of excess $i$,
  \item the tree $T$ produced by Algorithm \ref{algo:subtree} does not
    belong to the family $\left(\td i\right)_{0\leq i\leq d-2}$.
  \end{itemize}
\end{prop}

\smallskip

By Proposition~\ref{prop:order}, the sequence of maps on which
Algorithm \ref{algo:complet} is recursively performed is
decreasing, and hence converges eventually to the trivial map. 
Following the
sequence backwards provides a recursive construction of the \TCP of
$M$. A complete run of the algorithm for the pentagulation of
Fig.\ref{fig:oriPent} is shown in Figure~\ref{fig:aretebord}: Algorithm
\ref{algo:lineage} in Fig.\ref{subfig:lineage}, Algorithm
\ref{algo:subtree} in Fig.\ref{subfig:dfs}, the replacement of the
subtree in Fig.\ref{subfig:grafting}, and the recursive call
afterwards.

\begin{thm}\label{thm:linear} For any input map $M \in \Mdd$, Algorithm
  \ref{algo:complet} computes its blossoming forest in linear time
  with respect to the size of~$M$.
\end{thm}

\smallskip

%%%%%%%%%%%%%%%%%%%%%%%%%%%%%%%%%%%%%%%%%%%%%%%%%%%%%%%%%%%%%%%%%%%%%
\subsection{Identification of a subtree of $\BB{M}$}\label{sub:opepath}
%%%%%%%%%%%%%%%%%%%%%%%%%%%%%%%%%%%%%%%%%%%%%%%%%%%%%%%%%%%%%%%%%%%%%

In this subsection, we prove Proposition~\ref{prop:bernardi++}.  First
observe that, assuming Proposition~\ref{lem:path}, the proof of
Proposition~\ref{prop:bernardi++} is easily derived from the proof
of Proposition~\ref{prop:bernardi}, given in~\cite{Bernardi07}:
Algorithm~\ref{algo:subtree} is actually nothing but an adapted
version of the depth-first search process described in
Section~\ref{sec:effective}, in a slightly looser case.

\medskip

Therefore we only need to prove Proposition~\ref{lem:path}.  To this
purpose, we need the following technical result:
% , illustrated in Fig.~\ref{fig:deborde}:

\begin{defn} 
  Let $u$ and $w$ be two outer vertices of a blossoming map.  The
  \emph{clockwise distance} between $u$ and $w$ is the number of edge
  sides in the shortest clockwise path from $u$ to $w$ along the
  border of its outer face, see Fig.\ref{subfig:longueurl}.
\end{defn}

\begin{lem}\label{lem:deborde}
  Let $M\in \Mdd$ and $u,v$ in $M$ such that $v$ is a child of $u$ in
  \BB{M}.  Let $w$ be an outer vertex of $M$ that does not belong to
  \BB[u,v]{M}. Then, the clockwise distance between $u$ and $w$ in
  \MM{u,v} cannot be smaller than $\lceil d/2 \rceil$, see Fig.\ref{subfig:deborde}.
\end{lem}

\begin{proof}
  Let $T = \BB[u,v]{M}$, $n$ be its number of edges, and $0\leq i\leq
  d-3$ be the flow from $u$ to $v$. A simple combinatorial argument
  yields that the number of stems in $T$ is equal to
  $(2n+i)/(d-2)$. It implies that performing all the local closures of
  stems in $T$ requires $2n+i+1$ sides of edges (each stem needs $d-1$
  sides of edges, and each new created edge but the last
  one can be used in a further closure).
\smallbreak

  Now the number of sides of edges of $T$ is equal to $2n$, but not
  all of them can be used in local closures involving outgoing stems
  of $T$. Indeed the flow between a vertex in $T$ at distance $h$ of
  $u$ and its parent is at least $d-i-2h$, hence such a vertex can
  carry a stem only if $h> \lceil (d-2-i)/2\rceil$. Therefore the
  first stem explored during a DFS contour of $T$ is discovered after
  at least $\lceil (d-2-i)/2\rceil$ steps. Hence the number of sides
  of edges available for closures in $T$ is at most
  $2n-\lceil(d-2-i)/2\rceil$. It follows that closing the stems of $T$
  requires at least $\lceil d/2\rceil$ additional sides of edges;
  hence outer vertices of \MM{u,v} at clockwise distance less than
  $\lceil{d/2}\rceil$ from $u$ are separated from the outer face of
  $M$ by at least one closure edge.
\end{proof}

\begin{figure}[t]
  \centering \subfigure[The clockwise distance from $u$ to $w$ is equal to $3$.\label{subfig:longueurl}]{\includegraphics[page=1,
    scale=0.8]{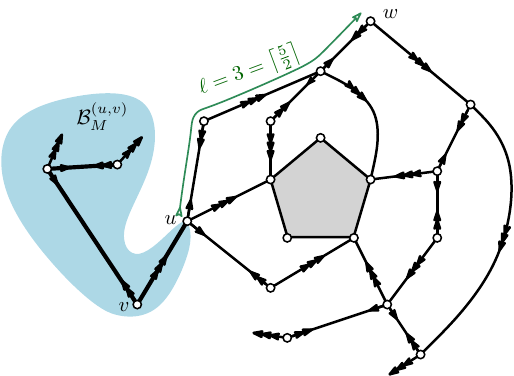}}\qquad \subfigure[For $w$ to be an outer
  vertex of $M$, it must be at distance at least $\lceil \frac{5}{2}\rceil = 3$ from $u$.\label{subfig:deborde}]{\includegraphics[page=2,
    scale =0.8]{deborde}}
  \caption{Illustration of Lemma~\ref{lem:deborde} for a pentagulation. }
  \label{fig:deborde}
\end{figure}

This lemma implies that if $(u,v)$ is a tree-edge and $u$ is the
parent of $v$, then the outer face cannot lie on the left of
$(v,u)$. In particular:

\begin{cor}
  Let $M\in\Mdd$.  Any outer clockwise saturated edge of $M$ (and
  there exists at least one of them) is a closure edge, and other outer
  edges (if any) have the outer face on their right when walking from
  a child to its parent.
\end{cor}

% We now describe the construction of a partial opening of $M$, which admits an
% outer lineage path, as illustrated on Figures~\ref{subfig:lineage} and~\ref{subfig:lineage2}:

\begin{proof}[Proof of Proposition~\ref{lem:path}]
  This enables to choose an outer closure edge $e$ to initialize
  Algorithm~\ref{algo:lineage}. Let $e_1, \dots e_\ell$ be the edges
  of the path returned by Algorithm~\ref{algo:lineage}, and let $v_0,
  \dots v_\ell$ be the corresponding vertices. Let us now prove by
  induction on the number of iterations that
  Algorithm~\ref{algo:lineage} correctly separates tree from closure
  edges. Assume that it is true after the $k$ first steps of the
  Algorithm~\ref{algo:lineage}, and let $i_k$ be the corresponding
  value of $i$. Assume that $i_{k}<l$, then $e = (v_{i_k}, w)$ for
  some vertex $w$ and $v_0$ is an outer vertex of $M$ at clockwise
  distance $i_k<d/2$ from $v_{i_k}$ around $M\backslash C$. By
  Lemma~\ref{lem:deborde}, it prevents $w$ from being a child of
  $v_{i_k}$. Hence if $e$ is saturated towards $v_i$, it implies that
  $e$ is a closure edge; otherwise $w$ is the parent of
  $v_{i_k}$. Moreover, from the construction, $w$ is necessarily an
  outer vertex of $M\backslash C$.  This ends the proof.
\end{proof}

\noindent
Let us end this section with two comments. 
\begin{rem} 
  The depth-first search process of Algorithm~\ref{algo:subtree} could
  have been applied on $e_1$ rather than $e_\ell$. But relying on this
  whole lineage path, we can identify a much bigger subtree.  This is
  actually a key point for proving the termination of recursive calls
  in Algorithm 3, as explained in Section~\ref{sec:smalltrees}.
\end{rem}

\smallbreak

\begin{rem}\label{rem:fin}
  If the lineage path reaches the root face, the depth-first search
  algorithm can be continued so as to identify \BB{M} entirely, as in
  Fig.\ref{subfig:dfs2}. Indeed in this case, let us open the
  closure edges pointing to the lineage path, collapse the root
  polygon into a single root vertex, and choose as root-corner of this
  new map the corner of the root vertex incident to $e_{\ell}$ and to
  the outer face. Then we are precisely in the setting of
  Proposition~\ref{prop:bernardi}.
\end{rem}

%%%%%%%%d%%%%%%%%%%%%%%%%%%%%%%%%%%%%%%%%%%%%%%%%%%%%%%%%%%%%%%%%%%%%%%
\subsection{Triangulations and quadrangulations}\label{sub:trigquad}
%%%%%%%%%%%%%%%%%%%%%%%%%%%%%%%%%%%%%%%%%%%%%%%%%%%%%%%%%%%%%%%%%%%%%
Simple triangulations and simple quadrangulations constitute a much
simpler case than general $d$-angulations of girth $d$, since all the edges are
saturated in their minimal \ddori. 

Given a $d$-angulation $M$ of girth $d$ with $d=3$ or $d=4$,
Algorithm~\ref{algo:lineage} identifies two tree-edges $e_1$ and
$e_2$, and Algorithm~\ref{algo:subtree} computes \MM{e_2}. \MM{e_2}
can be decomposed into \BB[e_2]{M} and \Mr{e_2}. Since all the edges
are saturated, the closure of \Mr{e_2} is itself a $d$-angulation,
smaller than $M$, on which Algorithm~\ref{algo:subtree} can
be iteratively applied.
  
At each iteration the number of edges of the map decreases, the sequence of
maps reaching eventually the trivial map reduced to a single cycle. Following
the sequence backwards provides a recursive construction of the tree-and-closure
partition of the desired map. Hence:

\begin{prop}
   For any simple triangulation or simple quadrangulation $M$,
   Algorithm~\ref{algo:triquad} computes \BB{M} in linear time.
 \end{prop}

\begin{algorithm}[h!]
\caption{BlossomingForest3Or4angulation(M)}\label{algo:triquad}
  \KwIn{a simple triangulation or quadrangulation $M$}
  \KwOut{its blossoming forest \BB M}
  \BlankLine
  let $T = \FuncSty{DetachableSubtree}(M)$\;
  delete $T$ in $M$ and perform the closure\;
  let $F = \FuncSty{BlossomingForest3Or4angulation}(M)$ \; 
  graft $T$ in $F$ in the appropriate corner\;
  \Return{F}\;
\end{algorithm}

%%%%%%%%%%%%%%%%%%%%%%%%%%%%%%%%%%%%%%%%%%%%%%%%%%%%%%%%%%%%%%%%%%%%%
\subsection{How to graft small subtrees}\label{sec:smalltrees}
%%%%%%%%%%%%%%%%%%%%%%%%%%%%%%%%%%%%%%%%%%%%%%%%%%%%%%%%%%%%%%%%%%%%%
This section is dedicated to the proof of the termination of
Algorithm~\ref{algo:complet} in the general case of $d$-angulations of
girth $d$ for $d\geq 5$.  Roughly speaking, we want to iterate
Algorithm~\ref{algo:subtree} on the closure of \Mr{e_\ell}, but this
map is not necessarily a $d$-angulation: this can happen if the edge
$e_\ell$ is not saturated, so that its deletion creates a vertex $u$
of outdegree less than $d$. To circumvent this problem,
Algorithm~\ref{algo:complet} grafts at $u$ a smaller
$d$-fractional subtree of appropriate excess, performs the closure,
and recursively calls itself on the resulting (smaller) map
in \Mdd. The existence of such a smaller subtree is
guaranteed by Proposition~\ref{prop:order}, which we shall now prove.

\medskip

A natural partial order on blossoming trees is the order induced by
sizes: a tree $t_1$ is declared to be smaller than a tree $t_2$ if
$t_1$ has less edges than $t_2$. But this order has to be
slightly refined so as to define unambiguously a unique minimal
element $\td i$ in each $\Tdn i$.  Precisely:

\begin{itemize}
\item we adopt the convention that the empty tree has excess $0$,
  hence it is equal to $\td 0$;\smallskip
\item for $d$ even or $i$ odd, there exists also a unique minimal
  element $\td i$ in $\Tdn i$, the one made of a path of length
  $(d-2-i)/2$ and one stem;\smallskip
\item for $d$ odd and $i$ even, exactly two trees have minimal size,
  both made of a path of length $d-2-i/2$ with two stems; consider
  their contour words on $\{b,e, \bar e\}$: they are respectively
  equal to $w_1 = e^p \; b \; e^q \; b \; \bar e^{(p+q)}$ and $w_2 =
    e^{p+q} \; b \; \bar e^q \; b \; \bar e^p$, with $p =
      \frac12(d-i-1)$ and $q=\frac12(d-3)$. We refine the ordering on
      $\Tdn i$ by defining the minimal element $\td i$ as the one with
      contour word~$w_1$.
\end{itemize}
\smallskip

This notion of ordering extends naturally to $d$-fractional forests:
$F_1<F_2$ if and only if $F_1$ can be obtained from $F_2$ by replacing
a sequence of subtrees by smaller ones. The maps in $\Mdd$ inherit the
ordering of their blossoming forests. The unique minimal map of $\Mdd$
is the map reduced to a simple cycle of length $d$.

\begin{proof}[Proof of Proposition~\ref{prop:order}]
  The case $i = 0$ is clear, and includes in particular the case where
  $v_\ell$ is a root vertex. From the exploration of the lineage path
  in Algorithm~\ref{algo:lineage}, any other tree that
  Algorithm~\ref{algo:subtree} may return has at least $d/2$ edges
  before the first stem in its contour word.  Since $\td i$ has less
  than $d/2$ edges for $d$ even or $i$ odd, this proves the result in
  those cases.  For $d$ odd and $i$ even, our particular choice for
  $\td i$ enables also to conclude.
\end{proof}

%%%%%%%%%%%%%%%%%%%%%%%%%%%%%%%%%%%%%%%%%%%%%%%%%%%%%%%%%%%%%%%%%%%%%
\subsection{Proof of Theorem~\ref{thm:linear} -- Complexity of the algorithm}
%%%%%%%%%%%%%%%%%%%%%%%%%%%%%%%%%%%%%%%%%%%%%%%%%%%%%%%%%%%%%%%%%%%%%

Propositions~\ref{lem:path} and~\ref{prop:bernardi++} prove that the
output of the algorithm is the \TCP of the input. It remains to prove
that the algorithm can be implemented in linear time.
\smallbreak

Let $(M_{k})$ be the sequence of maps obtained in the execution of the
algorithm. Observe that the length of the sequence $(M_k)$ is bounded
by $n/2 + 2n/(d-1)$; indeed, for each $k$, $M_{k+1}$ has at least two
edges less than $M_{k}$, except in the case where $\td i$ is grafted
instead of the other tree in $\Tdn i$ of minimal size; the latter case
happens at most $2n/(d-1)$ times, since any edge appears at most once
in such a subtree.
\smallbreak

The cost of step $k$ of the algorithm is clearly linear respectively
to the number of edges of $(M_k)$, since only exploration processes
are performed, with a bounded cost per visited edge.  Unfortunately,
this is not precise enough to evaluate correctly the total complexity
of the algorithm, and we need to look more carefully at the number of
visited edges of each type at each step.
The cost of step $k$ can be decomposed into three parts:

\begin{itemize}
\item the contribution $c(k)$ of the handling of closure edges
  discovered during the exploration of the lineage path, described in
  Algorithm~1;\smallskip
\item the contribution $t(k)$ of the exploration of the subtree
  $T_k$, including the computation of the lineage path, obtained by
  Algorithm~2;\smallskip
\item the contribution $n(k)$ of the construction of $M_{k+1}$,
  including the deletion of $ \BB[e_\ell]{M_k}$ and the grafting of
  the appropriate $\td i$, described in Section~\ref{sec:smalltrees}.
\end{itemize}

The number of edges in $T_k$ is the sum of the size of the adequate
$\td i$ -- which is bounded by $d$ -- and of the size difference
between $M_{k+1}$ and $M_k$. Hence, the sum over $k$ of $t(k)$ is
linear in $n$. Likewise the sum over $k$ of $n(k)$ is linear in $n$.
\smallbreak

To deal with $c(k)$, we need to be slightly more careful. Indeed, the
number of closure edges met during the exploration of the lineage path
at any given step $k$ can only be bounded by $n$, resulting in an upper
bound of order $n^2$ for the total complexity.  However, let us
observe that a bundle of closure edges that end at a same vertex of
the lineage path in $M_k$ also end at a same vertex in $M_{k+1}$
(provided that they still belong to $M_{k+1}$). Therefore, we assume
that $M$ is encoded in such a way that closure edges ending at the
same vertex are stored in the same object, called a \emph{bundle}, to
be used during the further exploration processes. In particular, the
number of bundles encountered during Algorithm~1 is bounded by $d$.

\begin{figure}[t]
  \centering
  \subfigure[\label{fig:algo2path}Identification of a lineage path by Algorithm~\ref{algo:lineage}.]{\includegraphics[page=2, scale=0.55]{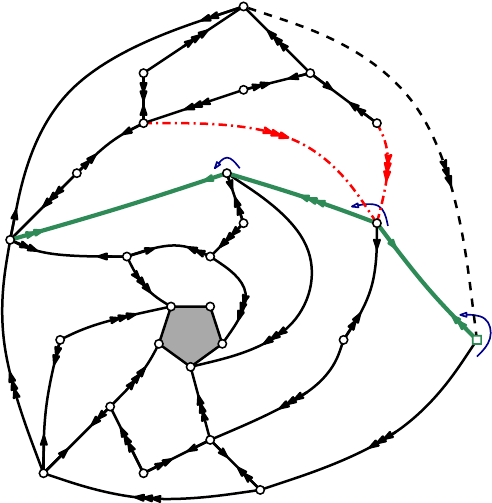}}\qquad
  \subfigure[\label{fig:algo2tree} Identification of the subtree $\mathcal B _{M}^{(e)}$.]{\includegraphics[page=3, scale =0.55]{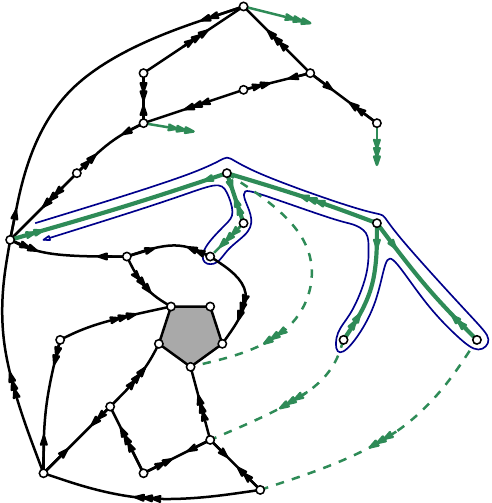}}\qquad
  \subfigure[\label{fig:algo2remove}The subtree $\mathcal B_{M}^{(e)}$ is replaced by $t_{2}^{(5)}$.]{\;\includegraphics[page=4, scale =0.55]{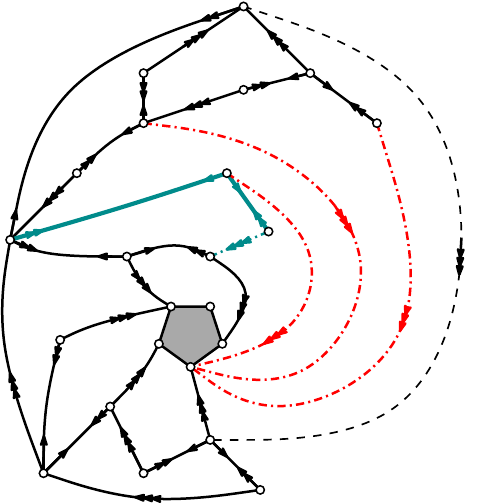}}
  \caption{Merging of a bundle (made of the two
    red dash-dotted edges in Fig.\subref{fig:algo2path}) \\with
    one of the closure edges of $t_{2}^{(5)}$ in
    Fig.\subref{fig:algo2remove}.}
  \label{fig:algo2}
\end{figure}

More precisely, in addition to the classical encoding of a map as a
collection of half-edges cyclically ordered around vertices and faces,
we introduce \emph{bundles} of half-edges. Each bundle contains a
doubly-linked list of consecutive closure half-edges ending at the
same vertex.  Each half-edge belongs to at most one bundle. At the
beginning of the algorithm, there is no bundle. Then, each time an
half-edge is identified as the extremity of a closure edge, it is
either added to an existing bundle, or used to initialize a new
one. If two bundles happen to become consecutive, they are
merged. This can occur during Algorithm~\ref{algo:lineage} or after
the grafting of the appropriate $\td i$ in
Algorithm~\ref{algo:complet}, see Fig.\ref{fig:algo2}.  The only other
operation to be performed on a bundle is the deletion of its
half-edges that belong to the subtree \BB[e_\ell]{M_k} identified by
Algorithm~2.

\smallskip
The total number of closure edges that need to be dealt with during
Algorithm~\ref{algo:complet} is linear in the number $n$ of vertices:
the number of closure edges of the initial $d$-angulation is linear in
$n$, and at each step, at most two new closure edges are created. The
merging of two bundles and the deletion of consecutive half-edges from
one bundle can be both implemented in constant time. Hence the sum
over $k$ of $c(k)$ is linear in~$n$.

\bibliographystyle{plain}

%%%%%%%%%%%%%%%%%%%%%%%%%%%%%%%%%%%%%%%%%%%%%%%%%%%%%%%%%%%%%%%%%%%%%
%%%%%%%%%%%%%%%%%%%%%%%%%%%%%%%%%%%%%%%%%%%%%%%%%%%%%%%%%%%%%%%%%%%%%
\end{document}